\documentclass[preprint]{elsarticle}
\usepackage{amssymb}
\usepackage{amsmath}
\usepackage{amsthm}
\usepackage{mathrsfs}
\usepackage{url}
\usepackage{hyperref}
\input xy
\xyoption{all}
\usepackage{enumerate}
\usepackage{multicol}
\usepackage{graphicx}
\usepackage{tikz}
\usepackage{tkz-graph}
\usepackage{framed}
\usepackage{array}
\usepackage{chngpage}
\usepackage{youngtab}
\usepackage[shortlabels]{enumitem}
\setitemize{wide}

\newtheorem{thm}{Theorem}[section]
\newtheorem{lemma}[thm]{Lemma}
\newtheorem{claim}[thm]{Claim}
\newtheorem{cor}[thm]{Corollary}
\newtheorem{prop}[thm]{Proposition}
\newtheorem{conj}[thm]{Conjecture}

\newtheorem{Definition}[thm]{Definition}
\newenvironment{definition}
  {\begin{Definition}\rm}{\end{Definition}}
\newtheorem{Example}[thm]{Example}
\newenvironment{example}
  {\begin{Example}\rm}{\end{Example}}
\newtheorem{Remark}[thm]{Remark}
\newenvironment{remark}
  {\begin{Remark}\rm}{\end{Remark}}

\title{Interlacing networks: birational RSK, the octahedron recurrence, and Schur function identities.}
\author{Miriam Farber\footnote{\emph{Email address}: \href{mailto:mfarber@mit.edu}{mfarber@mit.edu}. This author is supported in part by the NSF Graduate Research Fellowship grant 1122374.}}
\author{Sam Hopkins\footnote{\emph{Email address}: \href{mailto:shopkins@mit.edu}{shopkins@mit.edu}.}}
\author{Wuttisak Trongsiriwat\footnote{\emph{Email address}: \href{mailto:wuttisak@mit.edu}{wuttisak@mit.edu}.}}
\address{Massachusetts Institute of Technology, Cambridge MA, 02139}

\begin{document}

\begin{abstract}
Motivated by the problem of giving a bijective proof of the fact that the birational RSK correspondence satisfies the octahedron recurrence, we define interlacing networks, which are certain planar directed networks with a rigid structure of sources and sinks. We describe an involution that swaps paths in these networks and leads to Pl\"{u}cker-like three-term relations among path weights. We show that indeed these relations follow from the Pl\"{u}cker relations in the Grassmannian together with some simple rank properties of the matrices corresponding to our interlacing networks. The space of matrices obeying these rank properties forms the closure of a cell in the matroid stratification of the totally nonnegative Grassmannian. Not only does the octahedron recurrence for RSK follow immediately from the three-term relations for interlacing networks, but also these relations imply some interesting identities of Schur functions reminiscent of those obtained by Fulmek and Kleber. These Schur function identities lead to some results on Schur positivity for expressions of the form~$s_{\nu}s_{\rho} - s_{\lambda}s_{\mu}$.
\end{abstract}

\maketitle

\section{Introduction}

The Robinson-Schensted (RS) correspondence is a bijection between permutations $\sigma$ in the symmetric group $S_n$ and pairs $(P,Q)$ of standard Young tableaux of the same shape $\lambda \vdash n$. Under this correspondence, the first part~$\lambda_1$ of $\lambda$ has a simple interpretation as the size of the longest increasing subsequence of $\sigma$; more generally, the partial sum $\lambda_1 + \lambda_2 + \cdots + \lambda_k$ is equal to the maximum size of a union of $k$ disjoint increasing subsequences in $\sigma$. This description of the shape under RS of a permutation  in terms of longest increasing subsequences is known as Greene's theorem~\cite{greene}. The Robinson-Schensted-Knuth (RSK) correspondence is a generalization of the RS correspondence which takes arbitrary $n\times n$ $\mathbb{N}$-matrices $A$ to pairs $(P,Q)$ of semistandard Young tableaux of the same shape $\lambda$. In the RSK correspondence, the first part $\lambda_1$ has an analogous interpretation as the maximum weight of a path in $A$ from~$(1,1)$ to~$(n,n)$. Here the weight of a path is just the sum of the entries of the boxes it visits.  Similarly, the partial sum $\lambda_1 + \lambda_2 + \cdots + \lambda_k$ is equal to the maximum weight over $k$-tuples of noncrossing paths in $A$ connecting $(1,1),(1,2),\ldots,(1,k)$ to~$(n,n-k+1),(n,n-k+2),\ldots,(n,n)$. This extension of Greene's theorem to RSK is apparently folklore; the best reference we have for it is~\cite[Theorem~12]{krattenthaler}. For general background on RSK and its importance in the theory of symmetric functions, see~\cite[Chapter~7]{stanley2}.

Recently there has been significant interest in a birational lifting of RSK which takes matrices with entries in $\mathbb{R}_{> 0}$ to certain three-dimensional arrays with entries in $\mathbb{R}_{> 0}$; see, e.g.,~\cite{kirillov2},~\cite{noumi},~\cite{danilov1},~\cite{danilov3},~\cite{corwin},~\cite{oconnell1},~\cite{oconnell2}. The noncrossing paths interpretation of RSK has a direct analog in the birational setting (where in this process of detropicalization, maximums becomes sums and sums become products). Also it turns out that, subject to the proper renormalization, the output array of this birational map obeys the octahedron recurrence~\cite{danilov2},~\cite{danilov3}. For an excellent introduction to the octahedron recurrence and its appearance in various combinatorial problems, see~\cite{speyer}. This paper's main motivation is to provide a combinatorial proof of the fact that the sums over weighted tuples of noncrossing paths of the form encountered in the birational RSK correspondence obey the octahedron recurrence. Although this has been established already in~\cite{danilov3} by algebraic means, we present a direct, bijective proof, akin to the standard proof of the famous Lindstr\"{o}m--Gessel--Viennot lemma (see~\cite[Theorem~2.7.1]{stanley1}).

To this end, in \S\ref{sec:networks} we define ``interlacing networks'', which are certain planar directed networks with a rigid structure of sources and sinks. In~\S\ref{sec:sinkswap} we state and prove our main result concerning these networks: the existence of an involution that swaps pairs of tuples of noncrossing paths connecting sinks and sources. This involution leads to several three-term relations between the weights of these pairs of tuples of noncrossing paths that are akin to the three-term Pl\"{u}cker relations. Because of the well-known correspondence between totally-positive matrices and planar directed networks, these three-term relations are equivalent to determinantal identities for matrices of a particular form. We study these matrices in~\S\ref{sec:matrices}, where we give an alternative algebraic proof of these determinantal identities via the Pl\"{u}cker relations. Indeed, the identities follow from a simple rank property of these matrices together with the Pl\"{u}cker relations. The space of matrices obeying this rank property forms the closure of a cell in the matroid stratification of the positive Grassmannian. Thus we connect interlacing networks to the combinatorial theory of total positivity initiated by Postnikov~\cite{postnikov}.

In \S\ref{sec:octa} we return to our original motivation and show how the octahedron recurrence for birational RSK follows immediately from the three-term relations. In \S\ref{sec:schur} we give a rather different application of these three-term relations: namely, we show that they imply some interesting Schur function identities. These identities are reminiscent of those obtained by Fulmek and Kleber~\cite{fulmek} (and earlier by Kirillov~\cite{kirillov1}). Fulmek and Kleber also give a bijective proof of their identities by an explicit procedure that swaps pairs of tuples of noncrossing paths, though their swapping procedure and their identities differ in significant ways from our own. These kind of Schur function identities are closely related to questions of Schur positivity studied, for instance, in~\cite{fomin},~\cite{bergeron},~\cite{lam},~\cite{mcnamara}. In particular, we explain how these Schur function identities prove some special cases of a conjecture communicated to us by Alex Postnikov about general conditions under which symmetric functions of the form $s_{\nu}s_{\rho} - s_{\lambda}s_{\mu}$ are Schur positive. We go on to demonstrate some further Schur function identities obtained from the same involution defined in~\S\ref{sec:networks} by observing that this involution turns out to apply to a more general class of networks than the interlacing networks we were originally interested in studying. These additional Schur function identities are no longer three-term, but still lead to additional results about Schur positivity.

In short, the theory of interlacing networks lies at the intersection of many topics in modern algebraic combinatorics: RSK, the octahedron recurrence, the Lindstr\"{o}m--Gessel--Viennot method, total positivity, Schur functions, and Schur positivity.

\noindent {\bf Acknowledgements}: This paper developed out of some discussion about birational RSK during the weekly combinatorics preseminar held at MIT. We thank all those who participated in this preseminar and we especially thank the organizer Alex Postnikov. We also thank Darij Grinberg for his careful reading and for pointing out the paper~\cite{fulmek} to us. Finally, we thank the anonymous referees for many helpful comments and help with the references.

\section{Interlacing networks: definitions and notation} \label{sec:networks}

For two integers $m, n \in \mathbb{Z}$ we set~$[m,n] := \{m,m+1,\ldots,n\}$ (which is empty if $m > n$) and $[n] := [1,n]$. For~$X \subseteq \mathbb{Z}$ we set $X_{\mathrm{even}} := X \cap 2\mathbb{Z}$. For a finite set~$X$ and~$k \in \mathbb{N}$, we use $\binom{X}{k}$ to denote the set of subsets of~$X$ of cardinality $k$, and use $2^X := \cup_{k=0}^{\infty} \binom{X}{k}$ to denote the powerset of~$X$. For a tuple~$T = (t_1,\ldots,t_k)$ we sometimes write~$t \in T$ as though $T$ were a set to mean $t = t_i$ for some $i \in [k]$. Finally, if $A$ is a $m \times n$ matrix and $U \subseteq [m]$ and $W \subseteq [n]$, we denote by $A[U|W]$ the submatrix of~$A$ with rows indexed by $U$ and columns indexed by $W$.

For our purposes, a \emph{graph}  $G = (V,E,\omega)$ is a finite, directed, acyclic, planar, edge-weighted graph with vertex set~$V$, edge set~$E \subseteq V \times V$, and edge-weight function~$\omega \colon E \to \mathbb{R}_{>0}$. Let $G$ be a graph. A \emph{path} in~$G$ is a sequence~$\pi = \{v_i\}_{i=0}^{n}$ of distinct elements of~$V$ with~$(v_{i-1},v_i) \in E$ for all~$i \in [n]$. We say that such a path~$\pi$ \emph{connects}~$v_0$ and $v_n$, and that~$v_0$ is the \emph{start point} of $\pi$ and~$v_n$ is its \emph{end point}. We use $\mathrm{Vert}(\pi)$ to denote the set of vertices in $\pi$. The \emph{weight} of~$\pi$ is~$\mathrm{wt}(\pi) := \prod_{i=1}^{n} \omega(v_{i-1},v_i)$. A \emph{subpath} of $\pi$ is a subsequence of consecutive vertices. Sometimes we view paths as simple curves embedded in the plane in the obvious way. Two paths~$\pi$ and $\sigma$ are \emph{noncrossing} if~$\mathrm{Vert}(\pi) \cap \mathrm{Vert}(\sigma) = \emptyset$. Let~$\Pi =(\pi_1,\ldots,\pi_n)$ be a tuple of paths. We say~$\Pi$ is \emph{noncrossing} if $\pi_i$ and $\pi_j$ are noncrossing for all~$1 \leq i\neq j \leq n$. Suppose that~$\mathcal{X} = \{x_i\}_{i=1}^{n}$ and~$\mathcal{Y} = \{y_i\}_{i=1}^{n}$ are two sequences of vertices in~$V$ of the same length $n$. Then we denote the set of all $n$-tuples of noncrossing paths connecting $\mathcal{X}$ and $\mathcal{Y}$ by
$\mathrm{NCPath}_G(\mathcal{X},\mathcal{Y})$.
We omit the subscript $G$ when the network is clear from context. We define the \emph{weight} of the tuple $\Pi$ to be $\mathrm{wt}(\Pi) := \prod_{i=1}^{n} \mathrm{wt}(\pi)$. Recall that we are most interested in pairs of tuples of noncrossing paths, so for a pair~$(\Pi,\Sigma)$ of tuples of paths we define $\mathrm{wt}(\Pi,\Sigma) := \mathrm{wt}(\Pi)\cdot\mathrm{wt}(\Sigma)$.

Let $k \geq 2$ be some fixed constant. A \emph{network} is a triple $(G,S,T)$, where
\begin{itemize}
\item $G = (V,E,\omega)$ is a graph (in the sense above);
\item $S = (s_1,\ldots,s_{2k-1}) \in V^{2k-1}$ is a tuple of \emph{source} vertices;
\item $T = (t_1,\ldots,t_{2k-1}) \in V^{2k-1}$ is a tuple of \emph{sink} vertices,
\end{itemize}
such that $G$ is embedded inside a planar disc with~$s_1,\ldots,s_{2k-1}$, $t_{2k-1},\ldots,t_1$ arranged in clockwise order on the boundary of this disc. Note that both $S$ and $T$ are allowed to have repeated vertices. Needless to say, such networks are considered up to homeomorphism. We assume the edges of $G$ intersect the boundary of the disc into which $G$ is embedded only at vertices. In this section and the next we will work with a fixed network $(G,S,T)$; we will refer to this network from now on as simply $G$ with the sources and sinks implicit. Note in particular that the parameter~$k$ is therefore fixed in this section and the next. A \emph{pattern} on~$G$ is just a pair~$(I,J)$ with~$I,J \in \binom{[2k-1]}{k-1}$, where we think of $s_i$ and $t_j$ being colored red for all~$i\in I$ and~$j \in J$, and the other source and sink vertices being colored blue. We call $I$ the \emph{source pattern} of~$(I,J)$, and $J$ its \emph{sink pattern}. We will use $\mathrm{Pat}(G)$ to denote the set of patterns on $G$.

Let $I,J \in \binom{[2k-1]}{m}$ where the elements of~$I$ are~$i_1 < \cdots < i_{m}$ and the elements of $J$ are~$j_1 < \cdots < j_{m}$. Define the set of \emph{tuples of noncrossing paths of type $(I,J)$} to be $\mathrm{NCPath}_G(I,J) :=  \mathrm{NCPath}_G(\{s_{i_l}\}_{l=1}^{m}, \{t_{j_l}\}_{l=1}^{m})$. Now fix a pattern~$(I,J) \in \mathrm{Pat}(G)$. Define the set of \emph{pairs of tuples of noncrossing paths of type $(I,J)$} to be $\mathrm{PNCPath}_G(I,J) :=  \mathrm{NCPath}_G(I,J) \times  \mathrm{NCPath}_G(\overline{I},\overline{J})$, where for a subset $K \subseteq [2k-1]$ we set $\overline{K} := [2k-1] \setminus K$. Again, we omit the subscripts of $\mathrm{NCPath}_G(I,J)$ and $\mathrm{PNCPath}_G(I,J)$ when the network is clear from context. We then define the \emph{weight} of a pattern $(I,J) \in \mathrm{Pat}(G)$ to be~$\mathrm{wt}(I,J) := \sum_{(R,B) \in \mathrm{PNCPath}(I,J)} \mathrm{wt}(R,B)$. Denote the set of all pairs of tuples of noncrossing paths of $G$ by~$\mathrm{PNCPath}(G) := \bigcup_{(I,J) \in \mathrm{Pat}(G)}  \mathrm{PNCPath}(I,J)$.

We now define what it means for a network to be interlacing, the key property that will allow us to find three-term relations among the pattern weights. This condition may at first appear ad-hoc, but the later algebraic treatment of these networks will show that this definition suffices for the corresponding matrix to have a certain easily-stated rank property. Let us call~$U \subseteq V$ \emph{non-returning} if for all $u_1,u_2 \in U$ and paths~$\pi$ connecting $u_1$ and $u_2$, we have~$\mathrm{Vert}(\pi)\subseteq U$ (this is a technical condition required for our sink-swapping algorithm to work). Then we say $G$ is \emph{$k$-bottlenecked} if there exists a non-returning subset $N \subseteq V$ with~$|N| \leq k$ so that for all~$i,j \in [2k-1]$ and paths~$\pi$ connecting~$s_i$ to~$t_j$, there is~$v \in \mathrm{Vert}(\pi)$ for some~$v \in N$. Let us call~$U \subseteq V$ \emph{sink-branching} if for all~$u \in U$, $i \in [2,2k-2]$, $j \in \{1,2k-1\}$ and paths $\pi'$ connecting~$u$ and~$t_i$ and~$\pi''$ connecting~$u$ and~$t_j$, we have that~$\mathrm{Vert}(\pi') \cap \mathrm{Vert}(\pi'') = \{u\}$ (this is another technical condition). Then we say~$G$ is \emph{$(k-1)$-sink-bottlenecked} if there exists a non-returning and sink-branching~$N_T \subseteq V$ with $|N_T| \leq k-1$ so that for all~$i \in [2k-1], j \in [2,2k-2]$ and paths $\pi$ connecting~$s_i$ to~$t_j$, there is~$v \in \mathrm{Vert}(\pi)$ for some $v \in N_T$.  We say~$G$ is \emph{interlacing} if it is both $k$-bottlenecked and $(k-1)$-sink-bottlenecked.

\begin{example}
The following interlacing network, the rectangular grid $\Gamma_{m,n}^{k}$, will serve as our running example and will also be key for the motivating problem concerning birational RSK. Let $m, n \geq 3$ and $2 \leq k < \mathrm{min}(m,n)$. The graph~$\Gamma_{m,n}$ has vertex set $V := \{(i,j) \in \mathbb{Z}^2 \colon i \in [m], j \in [n]\}$ and edge set~$E := E_1 \cup E_2$ where
\begin{align*}
E_1 &:= \{ ( (i,j),(i+1,j) )\colon i \in [m-1], j \in [n] \}; \\
E_2 &:= \{ ( (i,j),(i,j+1) )\colon i \in [m], j \in [n-1] \}.
\end{align*}
We allow the weight function $\omega$ of the graph $\Gamma_{m,n}$ to be arbitrary. The network~$\Gamma_{m,n}^{k}$ has underlying graph $\Gamma_{m,n}$ with sources and sinks
\begin{align*}
S &:= ((k,1),(k-1,1),(k-1,2),(k-2,2),\ldots,(1,k-1),(1,k)) \\
T &:= ((m,n-k),(m-1,n-k),\cdots,(m-k,n-1),(m-k,n)).
\end{align*}
Our term ``interlacing network'' derives from the fact that these sinks and sources are arranged in a zig-zag.  Strictly speaking, in order to satisfy the network condition requiring our graph to lie inside a disc with the source and sink vertices on the boundary, we should restrict the vertex set of~$\Gamma_{m,n}^{k}$ to~$V' \subseteq V$ where $V' := \{v \in V\colon s_i \leq v \leq t_j \textrm{ for some $i,j \in [2k-1]$}\}$. However, vertices in $V \setminus V'$ will never be used in a path connecting a source to a sink, so this technicality will not concern us from now on.

Observe that $\Gamma_{m,n}^{k}$ is interlacing: we may take $N = \{s_1,s_3,\ldots,s_{2k-1}\}$ to satisfy the $k$-bottlenecked condition, and $N_T = \{t_2,t_4,\ldots,t_{2k-2}\}$ to satisfy the $(k-1)$-sink-bottlenecked condition. Figure~\ref{fig:intnetworkex} depicts $\Gamma_{8,10}^{4}$ along with an element of $\mathrm{PNCPath}(\{2,4,6\},\{2,4,6\})$. Note that vertex~$(1,1)$ is the top-leftmost vertex in this picture, $(8,1)$ is the bottom-leftmost vertex, and $(1,10)$ is the top-rightmost: we use ``matrix coordinates''' with edges directed downwards and rightwards in~$\Gamma_{m,n}^{k}$. We warn the reader that there are various conventions for orientation of such a grid, and we will at different times use several of them.
\end{example}

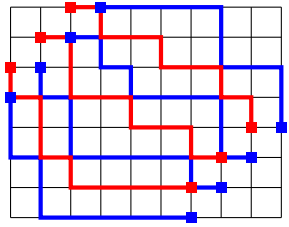
\begin{figure}
\begin{center}
\begin{tikzpicture}[scale=0.4,rotate = 270]
\foreach \x in {1,...,10}{
	\draw[thin] (1,\x) -- (8,\x);
}
\foreach \y in {1,...,8}{
	\draw[thin] (\y,1) -- (\y,10);
}
\draw[color=blue,ultra thick] (4,1) -- (6,1) -- (6,2) -- (8,2) -- (8,7);
\draw[color=blue,ultra thick] (3,2) -- (4,2) -- (4,3) -- (6,3) -- (6,7) -- (7,7) -- (7,8);
\draw[color=blue,ultra thick] (2,3) -- (2,4) -- (3,4) -- (3,5) -- (4,5) -- (4,8) -- (6,8) -- (6,9);
\draw[color=blue,ultra thick] (1,4) -- (1,8) -- (3,8) -- (3,10) -- (5,10);
\draw[color=red,ultra thick] (3,1) -- (4,1) -- (4,2) -- (6,2) -- (6,3) -- (7,3) -- (7,7);
\draw[color=red,ultra thick] (2,2) -- (2,3) -- (4,3) -- (4,5) -- (5,5) -- (5,7) -- (6,7) -- (6,8);
\draw[color=red,ultra thick] (1,3) -- (1,4) -- (2,4) -- (2,6) -- (3,6) -- (3,8) -- (4,8) -- (4,9) -- (5,9);
\foreach \x in {1,...,4} {
	 \node[rectangle,fill=blue,inner sep=2pt] at (5-\x,\x) {};
	 \node[rectangle,fill=blue,inner sep=2pt] at (9-\x,\x+6) {};
}
\foreach \x in {1,...,3} {
	 \node[rectangle,fill=red,inner sep=2pt] at (4-\x,\x) {};
	 \node[rectangle,fill=red,inner sep=2pt] at (8-\x,\x+6) {};
}
\end{tikzpicture}
\end{center}
\caption{$\Gamma_{8,10}^{4}$ and an element of $\mathrm{PNCPath}(\{2,4,6\},\{2,4,6\})$.} \label{fig:intnetworkex}
\end{figure}

\section{The sink-swapping involution} \label{sec:sinkswap}

In this section we obtain three-term Pl\"{u}cker-like relations between pattern weights of an interlacing network $G$ via an algorithmically-defined involution on~$\mathrm{PNCPath}(G)$ that swaps sink patterns. It turns out that this same involution makes sense even assuming only that $G$ is $k$-bottlenecked and in this case also leads to relations among pattern weights. The relations obtained when~$G$ is $k$-bottlenecked are weaker than for interlacing $G$, but they are nevertheless interesting (and will have applications to Schur function identities in \S\ref{sec:schur}).

\begin{definition}
For $J, J' \in \binom{[2k-1]}{k-1}$, we say that $J'$ is a \emph{swap} of $J$ if $J \cap J' = \emptyset$. Clearly the relation of being a swap is symmetric. If $J$ and $J'$ are swaps of one another, there is a unique element $j^* \in [2k-1] \setminus (J \cup J')$ and we call~$j^*$ their \emph{pivot}. We say $J'$ is a \emph{balanced swap} of $J$ if it is a swap of $J$ and their pivot $j^*$ is such that $|J \cap [j^*]| = |J' \cap [j^*]|$. We say that that $J'$ is a \emph{end swap} of $J$ if it is a swap of $J$ and their pivot $j^*$ is either $1$ or $2k-1$. Observe that~$J'$ being an end swap of $J$ implies it is a balanced swap of $J$. Define $\mathrm{bswap}(J)$ (resp. $\mathrm{eswap}(J)$) to be the set of balanced swaps (resp. end swaps) of $J$.
\end{definition}

Our goal in this section is to prove the following theorem and corollaries:

\begin{thm} \label{thm:tswap}
Suppose $G$ is $k$-bottlenecked. Then there is a weight-preserving involution~$\tau\colon \mathrm{PNCPath}(G) \to \mathrm{PNCPath}(G)$ with
\[\tau(\mathrm{PNCPath}(I,J)) \subseteq \bigcup_{J' \in \mathrm{bswap}(J)} \mathrm{PNCPath}(I,J'),\]
for all $(I,J) \in \mathrm{Pat}(G)$.

Suppose further that $G$ is interlacing. Then for all $(I,J) \in \mathrm{Pat}(G)$ we have
\[ \tau(\mathrm{PNCPath}(I,J)) \subseteq \bigcup_{J' \in \mathrm{eswap}(J)}\mathrm{PNCPath}(I,J'). \]
\end{thm}

\begin{cor}\label{cor:coroftswap1}
Suppose $G$ is $k$-bottlenecked. Fix a source pattern $I \in \binom{[2k-1]}{k-1}$. Fix some $K \subseteq [2k-1]_{\mathrm{even}}$ and set $K' :=[2k-1]_{\mathrm{even}} \setminus K$. Then
\[ \tau \bigg( \bigcup_{\substack{(I,J) \in \mathrm{Pat}(G) \\ J_{\mathrm{even}} = K}} \mathrm{PNCPath}(I,J) \bigg) = \bigcup_{\substack{(I,J') \in \mathrm{Pat}(G) \\ J'_{\mathrm{even}}= K'}} \mathrm{PNCPath}(I,J')\]
and thus
\[ \sum_{\substack{(I,J) \in \mathrm{Pat}(G), J_{\mathrm{even}} = K}} \mathrm{wt}(I,J) = \sum_{\substack{(I,J') \in \mathrm{Pat}(G), J'_{\mathrm{even}} = K'}} \mathrm{wt}(I,J').\]
\end{cor}

\begin{cor}\label{cor:coroftswap2}
Suppose $G$ is interlacing. Fix a source pattern $I \in \binom{[2k-1]}{k-1}$. Suppose that the sink pattern $J \in \binom{[2k-1]}{k-1}$ is such that $\{1,2k-1\} \cap J = \emptyset$. Define~$J' := [2,2k-1] \setminus J$ and $J'' := [1,2k-2] \setminus J$. Then
\[\tau(\mathrm{PNCPath}(I,J)) = \mathrm{PNCPath}(I,J') \cup \mathrm{PNCPath}(I,J'')\]
and thus
$\mathrm{wt}(I,J) = \mathrm{wt}(I,J') + \mathrm{wt}(I,J'').$
\end{cor}

\begin{remark}
An anonymous referee pointed out to us the paper of Danilov, Karzanov, and Koshevoy~\cite{danilov5} in which the authors classify quadratic Pl\"{u}cker-like relations for functions generated by network flows using nonintersecting paths as in the above theorem and corollaries. Their approach is similar to the classification of inequalities between products of two matrix minors obtained by Skandera~\cite{skandera}. However, the formulae of~\cite{danilov5} seem somewhat complicated to apply for our purposes because they are satisfied by all planar networks (or equivalently all totally nonnegative matrices); whereas by restricting to networks (or equivalently matrices) of a special interlacing form we derive much simpler formulae that directly apply to our motivating problem of showing that birational RSK satisfies the octahedron recurrence.
\end{remark}

Before we can describe the bijection $\tau$ we need a technical result about posets. First we recall some poset terminology. Let $(P,\leq)$ be a finite poset. For~$x,y \in P$, we write~$x < y$ to denote~$x \leq y$ and $x \neq y$, as is standard.  At some point we will require the notion of a downset; for a subset $P' \subseteq P$ the \emph{downset} of~$P'$ is the set of all $x \in P$ with $x \leq y$ for some $y \in P'$. Recall that a \emph{chain} in~$P$ is a subset~$C \subseteq P$ such that any two elements of $C$ are related, and an \emph{antichain} in~$P$ is a subset~$A \subseteq P$ such that no two elements of~$A$ are related. If~$|A| = k$, then we say $A$ is a \emph{$k$-antichain}. Suppose that $m$ is the maximal size of an antichain in $P$. In this case, we can find a partition of $P$ into chains $C_1,\ldots,C_m$: that is, each~$C_i$ is a chain and we have~$\cup_{i=1}^{m} C_i = P$ and $C_i \cap C_j = \emptyset$ for~$i \neq j$. (This result is known as Dilworth's theorem; see for example Freese~\cite{freese}, who proves not only that the poset of maximal size antichains has a minimum, as we show below, but also that this poset is in fact a lattice.) Let~$\mathcal{A}_m(P)$ denote the set of $m$-antichains of~$P$. For any~$A \in \mathcal{A}_m(P)$ we must have that~$|A \cap C_i| = 1$ for all~$i \in [m]$. Thus we can define the following partial order on $\mathcal{A}_m(P)$: for two $m$-antichains~$X = \{x_i\}_{i=1}^{m}$ and $Y = \{y_i\}_{i=1}^{m}$ such that~$X \cap C_i = \{x_i\}$ and~$Y \cap C_i = \{y_i\}$ for all $i \in [m]$, we say that~$X \leq Y$ if and only if~$x_i \leq y_i$ for all~$i \in [m]$.

\begin{prop} \label{prop:antichain}
The poset $\mathcal{A}_m(P)$ has a minimum.
\end{prop}
\begin{proof} Given any $X, Y \in \mathcal{A}_m(P)$ which are incomparable, we claim that there is~$Z \in \mathcal{A}_m(P)$ so that $Z \leq X$ and $Z \leq Y$. Define~$z_i := \mathrm{min}(x_i,y_i)$ for all $i \in [m]$ and set $Z := \{z_1,\ldots,z_m\}$. Note that $Z$ is still an antichain: if $z_i < z_j$, then $\mathrm{min}(x_i,y_i) < \mathrm{min}(x_j,y_j)$, which means $\mathrm{min}(x_i,y_i) < x_j$ and~$\mathrm{min}(x_i,y_i) < y_j$, which forces a relation in $X$ or in $Y$. Because~$\mathcal{A}_m(P)$ is evidently finite,  and by definition nonempty, it has a minimum. \end{proof}

The order we defined on $\mathcal{A}_m(P)$ above in principle depended on the choice of chains $\{C_1,\ldots,C_m\}$; but in fact we can give a description of this order which does not depend on such a choice. Namely, for $X,Y \in \mathcal{A}_m(P)$, let us say~$X \leq' Y$ if for each $x \in X$ there exists $y \in Y$ such that $x \leq y$. Then~$X \leq' Y$ if and only if~$X \leq Y$. The implication $X \leq Y \Rightarrow X \leq' Y$ is trivial. To see~$X \leq' Y \Rightarrow X \leq Y$, write $X = \{x_i\}_{i=1}^{m}$ and $Y = \{y_i\}_{i=1}^{m}$ with~$X \cap C_i = \{x_i\}$ and $Y \cap C_i = \{y_i\}$ for all $i \in [m]$. Then for $i \in [m]$, we have that there is some~$y_j$ such that $x_i \leq y_j$. If $i = j$, then we are okay. So suppose $i \neq j$. It cannot be that $y_i \leq x_i$ as then $y_i \leq y_j$ and $Y$ would fail to be an antichain. But~$C_i$ is a chain, so this means $x_i < y_i$. Therefore we have~$X \leq Y$.

We now define the poset of intersections of a tuple of paths in $G$, which will be key in defining the bijection $\tau$ of Theorem~\ref{thm:tswap}. Let $\Pi = (\pi_1,\ldots,\pi_{m})$ be a tuple of paths in $G$. Then define $\mathrm{Int}^{\Pi} := \cup_{i \neq j} \mathrm{Vert}(\pi_i) \cap \mathrm{Vert}(\pi_j)$ to be the set of all intersections between paths in $\Pi$. First of all, we give~$\mathrm{Int}^{\Pi}$ a labeling function $\ell^{\Pi}\colon \mathrm{Int}^{\Pi} \to \mathcal{P}(\{\pi_1,\ldots,\pi_m\})$, whereby $\ell^{\Pi}(u) := \{\pi_i\colon u \in \mathrm{Vert}(\pi_i)\}.$ Secondly, we give $\mathrm{Int}^{\Pi}$ a partial order $\leq$ as follows. For a path $\pi_i  = \{v_j\}_{j=0}^{n}$, if~$v_i, v_j \in \mathrm{Int}^{\Pi}$ for $0 \leq i \leq j \leq n$ we declare~$v_j \preceq v_i$. We then define $\leq$ to be the transitive closure of~$\preceq$. It is routine to verify that $\leq$ indeed defines a partial order on~$\mathrm{Int}^{\Pi}$ (but note that here we use the acyclicity of $G$ in an essential way).

We need just a little more terminology related to paths in order to define~$\tau$. For a tuple $\Pi = (\pi_1,\ldots,\pi_n)$ of paths, let us say a vertex $v \in V$ is a \emph{$2$-crossing} of $\Pi$ if $|\{i\colon v \in \mathrm{Vert}(\pi_i)\}| = 2$. Let $v$ be a $2$-crossing of $\Pi$ and suppose that~$v \in \pi_i \cap \pi_j$ for $i \neq j$. Say $\pi_i = \{u_1,\ldots,u_a,v,u_{a+1},\ldots,u_b\}$ and~$\pi_j = \{w_1,\ldots,w_c,v,w_{c+1},\ldots,w_d\}$. Then define the \emph{flip of $\Pi$ at $v$} to be~$\mathrm{flip}_{v}(\Pi) := (\pi'_1,\ldots,\pi'_n)$ where
\[ \pi'_m := \begin{cases} \{u_1,\ldots,u_a,v,w_{c+1},\ldots,w_d\} &\textrm{ if $m = i$;} \\
 \{w_1,\ldots,w_c,v,u_{a+1},\ldots,u_b\} &\textrm{ if $m = j$;} \\
 \pi_m, &\textrm{otherwise.} \end{cases}\]
 For any two $2$-crossings $u,v$ of $\Pi$, we have $\mathrm{flip}_u  (\mathrm{flip}_v(\Pi)) = \mathrm{flip}_v(\mathrm{flip}_u(\Pi))$. Thus for a set $U = \{u_1,\ldots,u_p\} \subseteq V$ of $2$-crossings of $\Pi$, let us define the \emph{flip of~$\Pi$ at $U$} to be $\mathrm{flip}_U(\Pi) = \mathrm{flip}_{u_1}(\mathrm{flip}_{u_2}(\cdots \mathrm{flip}_{u_p}(\Pi) \cdots ) )$, where the composition may be taken in any order. Finally, for two tuples of paths~\mbox{$\Pi = (\pi_1,\ldots,\pi_m)$} and~\mbox{$\Sigma = (\sigma_1,\ldots,\sigma_n)$}, set~$\Pi + \Sigma := (\pi_1,\ldots,\pi_m,\sigma_1,\ldots,\sigma_n)$.

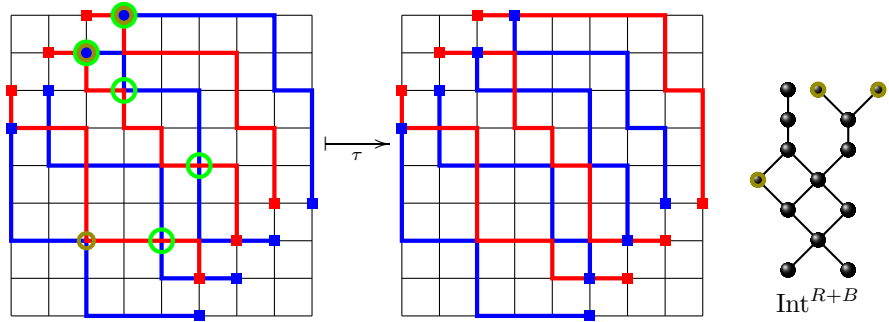
\begin{figure}
\begin{center}
\xymatrix {\begin{tikzpicture}[scale=0.5,rotate = 270]
\foreach \x in {1,...,9}{
	\draw[thin] (1,\x) -- (9,\x);
}
\foreach \y in {1,...,9}{
	\draw[thin] (\y,1) -- (\y,9);
}
\draw[color=blue,ultra thick] (4,1) -- (7,1) -- (7,3) -- (9,3) -- (9,6);
\draw[color=blue,ultra thick] (3,2) -- (5,2) -- (5,5) -- (8,5) -- (8,7);
\draw[color=blue,ultra thick] (2,3) -- (2,4) -- (3,4) -- (3,6)  -- (7,6) -- (7,8);
\draw[color=blue,ultra thick] (1,4) -- (1,8) -- (3,8) -- (3,9) -- (6,9);
\draw[color=red,ultra thick] (3,1) -- (4,1) -- (4,3) -- (7,3) -- (7,6) -- (8,6);
\draw[color=red,ultra thick] (2,2) -- (2,3) -- (3,3) -- (3,4) -- (4,4) -- (4,5) -- (5,5) -- (5,7) -- (7,7);
\draw[color=red,ultra thick] (1,3) -- (1,4) -- (2,4) -- (2,7) -- (4,7) -- (4,8) -- (6,8);
\foreach \x in {1,...,4} {
	 \node[rectangle,fill=blue,inner sep=2pt] at (5-\x,\x) {};
	 \node[rectangle,fill=blue,inner sep=2pt] at (10-\x,\x+5) {};
}
\foreach \x in {1,...,3} {
	 \node[rectangle,fill=red,inner sep=2pt] at (4-\x,\x) {};
	 \node[rectangle,fill=red,inner sep=2pt] at (9-\x,\x+5) {};
}
\draw[ultra thick,color=green]  (2,3) circle (0.3);
\draw[ultra thick,color=green]  (1,4) circle (0.3);
\draw[ultra thick,color=green]  (3,4) circle (0.3);
\draw[ultra thick,color=green]  (5,6) circle (0.3);
\draw[ultra thick,color=green]  (7,5) circle (0.3);
\draw[ultra thick,color=olive]  (2,3) circle (0.2);
\draw[ultra thick,color=olive]  (1,4) circle (0.2);
\draw[ultra thick,color=olive]  (7,3) circle (0.2);
\end{tikzpicture}   \ar@{|->}[r]<15ex>_-{\tau} & \begin{tikzpicture}[scale=0.5,rotate = 270]
\foreach \x in {1,...,9}{
	\draw[thin] (1,\x) -- (9,\x);
}
\foreach \y in {1,...,9}{
	\draw[thin] (\y,1) -- (\y,9);
}
\draw[color=blue,ultra thick] (4,1) -- (7,1) -- (7,3) -- (9,3) -- (9,6);
\draw[color=blue,ultra thick] (3,2) -- (5,2) -- (5,5) -- (7,5) -- (7,6) -- (8,6);
\draw[color=blue,ultra thick] (2,3) -- (3,3) -- (3,6) -- (5,6) -- (5,7) -- (7,7);
\draw[color=blue,ultra thick] (1,4) -- (2,4) -- (2,7) -- (4,7) -- (4,8) -- (6,8);
\draw[color=red,ultra thick] (3,1) -- (4,1) -- (4,3) -- (7,3) -- (7,5) -- (8,5) -- (8,7);
\draw[color=red,ultra thick] (2,2) -- (2,3) -- (2,4) -- (3,4) -- (4,4) -- (4,5) -- (5,5) -- (5,6) -- (7,6) -- (7,8);
\draw[color=red,ultra thick] (1,3)  -- (1,8) -- (3,8) -- (3,9) -- (6,9);
\node[rectangle,fill=blue,inner sep=2pt] at (4,1) {};
\node[rectangle,fill=blue,inner sep=2pt] at (9,6) {};
\foreach \x in {2,...,4} {
	 \node[rectangle,fill=blue,inner sep=2pt] at (5-\x,\x) {};
	 \node[rectangle,fill=red,inner sep=2pt] at (10-\x,\x+5) {};
}
\foreach \x in {1,...,3} {
	 \node[rectangle,fill=red,inner sep=2pt] at (4-\x,\x) {};
	 \node[rectangle,fill=blue,inner sep=2pt] at (9-\x,\x+5) {};
}
\end{tikzpicture} & \hspace{-0.5cm}
\begin{tikzpicture}[scale=0.4]
\SetVertexMath
\GraphInit[vstyle=Art]
\SetUpVertex[MinSize=3pt]
\SetVertexLabel
\tikzset{VertexStyle/.style = {
shape = circle,
shading = ball,
ball color = black,
inner sep = 2pt
}}
\SetUpEdge[color=black]
\Vertex[x=4,y=0,NoLabel]{v_1}
\Vertex[x=6,y=0,NoLabel]{v_2}
\Vertex[x=5,y=1,NoLabel]{v_3}
\Vertex[x=4,y=2,NoLabel]{v_4}
\Vertex[x=6,y=2,NoLabel]{v_5}
\Vertex[x=3,y=3,NoLabel]{v_6}
\Vertex[x=5,y=3,NoLabel]{v_7}
\Vertex[x=4,y=4,NoLabel]{v_8}
\Vertex[x=6,y=4,NoLabel]{v_9}
\Vertex[x=4,y=5,NoLabel]{v_{10}}
\Vertex[x=6,y=5,NoLabel]{v_{11}}
\Vertex[x=4,y=6,NoLabel]{v_{12}}
\Vertex[x=5,y=6,NoLabel]{v_{13}}
\Vertex[x=7,y=6,NoLabel]{v_{14}}
\Edges(v_1,v_3,v_4,v_6,v_8,v_{10},v_{12});
\Edges(v_2,v_3,v_5,v_7,v_9,v_{11},v_{13});
\Edges(v_4,v_7,v_8);
\Edges(v_{11},v_{14});
\draw[ultra thick,color=olive]  (3,3) circle (0.2);
\draw[ultra thick,color=olive]  (5,6) circle (0.2);
\draw[ultra thick,color=olive]  (7,6) circle (0.2);
\node at (5,-1) {$\mathrm{Int}^{R+B}$};
\end{tikzpicture}}
 \end{center}
\caption{Example~\ref{ex:tswapalg} for the involution $\tau$: the interlacing network is $\Gamma_{9,9}^{4}$; on the left is~$(R,B) \in \mathrm{PNCPath}(\{2,4,6\},\{2,4,6\})$ and~$\tau(R,B) \in \mathrm{PNCPath}(\{2,4,6\},\{3,5,7\})$ is on the right. The intersection poset $\mathrm{Int}^{R+B}$ is depicted far right. } \label{fig:tswapex}
\end{figure}

We proceed to define the involution $\tau$. So we assume from now on that~$G$ is~$k$-bottlenecked. Let $N \subseteq V$ be the subset guaranteed by the~$k$-bottlenecked property of $G$. If $|N| < k$, then $\mathrm{PNCPath}(G) = \emptyset$; so we may assume \mbox{$|N| = k$}. Let~$(I,J) \in \mathrm{Pat}(G)$ and let $(R,B) \in \mathrm{PNCPath}(I,J)$. Here we use $R$ for ``red'' and $B$ for ``blue'' as the example below will make clear. Say~$R = (r_1,\ldots,r_{k-1})$ and $B = (b_1,\ldots,b_k)$ and set~$\Pi := R+B$. Because $N$ is non-returning there is a subset of $N$ of size $k-1$ consisting of $2$-crossings of~$\Pi$ which in fact is a $(k-1)$-antichain of~$\mathrm{Int}^{\Pi}$. It is also clear that there is no antichain of size greater than $k-1$: indeed, given an antichain $U$ of $\mathrm{Int}^{\Pi}$ and any two elements~$u,v \in U$, for each $r \in R$ we have that~$r \in \ell^{\Pi}(u) \Rightarrow r \notin \ell^{\Pi}(v)$; but on the other hand, for any~$u \in U$, there must be some $r \in R$ with $r \in \ell^{\Pi}(u)$. So by Proposition~\ref{prop:antichain}, we conclude that $\mathcal{A}_{k-1}(\mathrm{Int}^{\Pi})$ has a minimum. Starting with this minimum antichain, we define $\tau(R,B)$ by the algorithm below.

\begin{framed}
\begin{center}
{\bf Algorithm defining $\tau$}
\end{center}
\noindent {\sc initialization}:\\
\indent Let $U \subseteq \mathrm{Int}^{\Pi}$ be the minimum of $\mathcal{A}_{k-1}(\mathrm{Int}^{\Pi})$. \\
\indent Let $\mathrm{FLIP}_0 := U$. \\
\indent Initialize the counter $c$ to $0$. \\
\indent Let $b_{n_c}$ be the  unique $b_i$ with $b_{n_c} \notin \ell^{\Pi}(u)$ for all $u \in U$. \\
\indent If there is $w$ in the downset of $U$ with $b_{n_c} \in \ell^{\Pi}(w)$: \\
\indent \indent Let $w_{c+1}$ be maximal in the downset of $U$ with $b_{n_c} \in \ell^{\Pi}(w_{c+1})$. \\
\indent \indent Increment the counter $c$ by $1$. \\
\indent \indent Enter the loop. \\
\indent Else: \\
\indent \indent Skip the loop. Proceed directly to output. \\
\noindent {\sc loop}: \\
\indent Let $r_{m_c}$ be the unique $r_i$ with $r_{m_c} \in \ell^{\Pi}(w_c)$. \\
\indent Let $v_c$ be minimal in $\mathrm{Int}^{\Pi}$ with $w_c < v_c$ and $r_{m_c} \in \ell^{\Pi}(v_c)$. \\
\indent Let $\mathrm{FLIP}_{c} := \mathrm{FLIP}_{c-1} \, \Delta \{v_c,w_c\}$ (with $\Delta =$ ``symmetric difference''). \\
\indent Let $b_{n_c}$ be the unique $b_i$ with $b_{n_c} \in \ell^{\Pi}(v_c)$. \\
\indent If there is $w$ with $b_{n_c} \in \ell^{\Pi}(w)$ and $w < v_c$: \\
\indent \indent Let $w_{c+1}$ be maximal in $\mathrm{Int}^{\Pi}$ with $w_{c+1} < v_c$ and $b_{n_c} \in \ell^{\Pi}(w_{c+1})$. \\
\indent \indent Increment the counter $c$ by $1$. \\
\indent \indent Return to the beginning of the loop. \\
\indent Else: \\
\indent \indent Exit the loop. \\
\noindent {\sc output}: \\
\indent Define $\tau(R,B) := (R',B')$ where $R' + B' := \mathrm{flip}_{\mathrm{FLIP}_c}(\Pi)$.
\end{framed}

\begin{example} \label{ex:tswapalg}
Before we prove the correctness of this algorithm, we give an example run of it. Let our network be $\Gamma_{9,9}^{4}$ and consider the pair of tuples of noncrossing paths $(R,B) \in \mathrm{PNCPath}(\{2,4,6\},\{2,4,6\})$ depicted in Figure~\ref{fig:tswapex}. Suppose $R = (r_1,\ldots,r_3)$ and $B = (b_1,\ldots,b_4)$ so the paths are labeled in left-to-right order in the figure. To apply $\tau$ to~$(R,B)$, first we find the minimum $3$-antichain $U$ in $\mathrm{Int}^{\Pi}$ where $\Pi := R+B$. This vertices in this antichain are circled by small olive-colored circles in Figure~\ref{fig:tswapex}, and the poset $\mathrm{Int}^{\Pi}$ is depicted to the right in the figure. In this case it turns out that~$U = \{(7,3),(2,3),(1,4)\}$. We initialize $\mathrm{FLIP}_0 := \{(7,3),(2,3),(1,4)\}$ and find that $n_0 = 2$. There is some~$w$ in the downset of $U$ with $b_2 \in \ell^{\Pi}(w)$, so we set $w_1 := (5,5)$ and enter the loop.
\begin{enumerate}
\item We find $m_1 = 2$ and $v_1 = (3,4)$, and we set
\[\mathrm{FLIP}_1 := \{(7,3),(2,3),(1,4),(5,5),(3,4)\}.\]
We find $n_1 = 3$ and there is $w$ with $b_3 \in \ell^{\Pi}(w)$ and $w < v_1$, so we set~$w_2 := (5,6)$ and enter the loop again.
\item We find $m_2 = 2$ and $v_2 = (5,5)$, and we set
\[ \mathrm{FLIP}_2 := \{(7,3),(2,3),(1,4),(3,4),(5,6)\}.\]
We find $n_2 = 2$ and there is $w$ with $b_2 \in \ell^{\Pi}(w)$ and $w < v_2$, so we set~$w_3 := (7,5)$ and enter the loop again.
\item We find $m_3 = 1$ and $v_3 = (7,3)$, and we set
\[ \mathrm{FLIP}_3 := \{(2,3),(1,4),(3,4),(5,6),(7,5)\}.\]
We find $n_3 = 1$ and there is no $w$ with $b_1 \in \ell^{\Pi}(w)$ and $w < v_3$, so we exit the loop.
\end{enumerate}
Finally, we define $\tau(R,B) := (R',B')$ where $R' + B' := \mathrm{flip}_{\mathrm{FLIP}_3}(\Pi)$. The elements of $\mathrm{FLIP}_3$ are circled by large light green circles in Figure~\ref{fig:tswapex} and $\tau(R,B)$ is shown to the right of~$(R,B)$. Note that $\tau(R,B) \in \mathrm{PNCPath}(\{2,4,6\},\{3,5,7\})$ and this is consistent with Theorem~\ref{thm:tswap} because~$\{3,5,7\} \in \mathrm{eswap}(\{2,4,6\})$.
\end{example}

We proceed to verify the correctness of the algorithm defining $\tau$. In the following series of claims we refer to the variables defined above in the description of the algorithm. In particular, $c$ refers to the value of the counter at the end of a run of the algorithm. Also, say~$R' = (r'_1,\ldots,r'_{k-1})$ and~$B' = (b'_1,\ldots,b'_{k})$ and set $\Pi' := R' + B'$. Some of this analysis is tedious but it is all necessary.

\begin{claim} \label{claim:algwelldef}
 For $i \in [c]$, there exists a unique $u_i$ with $w_i < u_i$ and~$r_{m_i} \in \ell^{\Pi}(u_i)$ such that $u_i \in U$.
\end{claim}
\begin{proof} This claim is required for the algorithm to make sense because it shows that $v_i$ is always well-defined. This claim also shows that $v_i$ always belongs to the downset of $U$. Observe that the uniqueness is trivial because for any~$r \in R$ there is a unique $u \in U$ with $r \in \ell^{\Pi}(u)$. So existence is what is at issue. We prove this claim by induction on $i$. For $i = 1$ it is clear because~$w_1 \leq u_1$ by definition, where~$u_1$ is the unique element of $U$ with $r_{m_1} \in \ell^{\Pi}(u_1)$, and~$w_1 \neq u_1$ because~\mbox{$b_{n_0} \in \ell^{\Pi}(w_1)$} but $b_{n_0} \notin \ell^{\Pi}(u_1)$. So suppose $i > 1$ and the claim holds for smaller $i$. Then assume $w_i \geq u_i$ where~$u_i$ is the unique element of $U$ with $r_{m_i} \in \ell^{\Pi}(u_i)$. Note that~$w_i < v_{i-1}$ and $v_{i-1} \leq u_{i-1}$ by our inductive assumption. Thus we conclude~$u_i < u_{i-1}$. But this contradicts the fact that $U$ is an antichain. So in fact $w_i < u_i$. The claim follows by induction. \end{proof}

\begin{claim} \label{claim:algterminates}
The algorithm terminates.
\end{claim}
\begin{proof} We claim it is impossible that $w_i = w_j$ for $i < j$. If $i > 1$, then~$w_i = w_j$ implies $w_{i-1} = w_{j-1}$. So suppose $i=1$. Then $w_1 = w_j$ for some~$j > 1$ implies that~$v_{j-1} > w_1$ with $b_{n_0} \in \ell^{\Pi}(v_{j-1})$. But $v_{j-1}$ is in the downset of~$U$ and $w_1$ was chosen to be maximal in the downset of~$U$ such that~\mbox{$b_{n_0} \in \ell^{\Phi}(w_1)$}, which is a contradiction. So indeed $w_i \neq w_j$ for all~$i \neq j$. Therefore, the algorithm terminates since $\mathrm{Int}^{\Pi}$ is finite. \end{proof}

\begin{claim} \label{claim:noncrossing}
The tuples of paths $R'$ and $B'$ are noncrossing.
\end{claim}
\begin{proof} For $0 \leq i \leq c$, let us define $R'_i$ and $B'_i$ by $R'_i+ B'_i := \mathrm{flip}_{\mathrm{FLIP}_i}(\Pi)$. For such $i$, set $\Pi'_i := R'_i + B'_i$. For all such $i$ we have $\mathrm{Int}^{\Pi} = \mathrm{Int}^{\Pi'_i}$ as unlabeled posets. Thus, the only way that  one of~$R'_i$ or $B'_i$ could fail to be noncrossing is if there were $w \in \mathrm{Int}^{\Pi'_i}$ where $w$ is a $2$-crossing of $R'_i$ or of $B'_i$. Let us say that such a $w$ is \emph{bad at step $i$}. We claim that if there is a $w$ which is bad at step $i$, then~$i < c$. Clearly this proves the claim because $R' = R'_c$, $B' = B'_c$, and~$\Pi' = \Pi'_c$.

If $w$ is bad at any step then $w$ must belong to the downset of $U$ as otherwise~$w$ would be above every vertex that we flip at. So we restrict our attention to $w$ in the downset of $U$. For $0 \leq i \leq c$ and $w$ in the downset of~$U$, we claim the following:
\begin{enumerate}
\item $r \in \ell^{\Pi}(w)$ for any $r \in R$ implies $b' \in \ell^{\Pi'_i}(w)$ for some~$b' \in B'_i$;
\item $b \in \ell^{\Pi}(w)$ for any $b \neq b_{n_i} \in B$ implies $r' \in \ell^{\Pi'_i}(w)$ for some~$r' \in R'_i$;
\item if $i >0$ and $w \geq v_i$, then $b_{n_i} \in \ell^{\Pi}(w)$ implies $r' \in \ell^{\Pi'_i}(w)$ for some~$r' \in R'_i$.
\end{enumerate}
This claims implies that if there is a $w$ which is bad at step $i$, then $i < c$ as this claim shows such a $w$ would have to have $b_{n_i} \in \ell^{\Pi}(w)$, with $w < v_i$ if~$i > 0$. We prove this claim by induction on $i$. For $i=0$, (1) and (2) hold because we flip at $U$ to obtain $\Pi'_0$, and (3) does not apply. So now assume that~$i > 0$ and the claim holds for $i-1$. Note that $\Pi'_i$ is obtained from $\Pi'_{i-1}$ by flipping at $v_i$ and $w_i$. Therefore the labels of the intersection poset only change for vertices that are below $w_i$ or $v_i$. The effect is that for $w$ in the downset of $U$ we have the following:
\begin{itemize}
\item $b_{n_{i-1}} \in \ell^{\Pi}(w)$ now implies $r' \in \ell^{\Pi'_i}(w)$ for some $r' \in R'_i$ even if $w \leq w_i$;
\item if $w \leq v_i$, $b_{n_i} \in \ell^{\Pi}(w)$ may no longer imply $r' \in \ell^{\Pi'_i}(w)$ for any $r' \in R'_i$;
\item but if $w \geq v_i$, $b_{n_i} \in \ell^{\Pi}(w)$ still implies $r' \in \ell^{\Pi'_i}(w)$ for some $r' \in \Pi'_i$.
\end{itemize}
Note that $n_i = n_{i-1}$ is possible, but this is not an issue because $w_i < v_i$. So~(1), (2), and (3) above hold for $i$, and the claim follows by induction. \end{proof}

\begin{claim} \label{claim:involution}
The map $\tau$ is an involution.
\end{claim}
\begin{proof} Suppose we run the algorithm again on $(R',B')$. Let us use primes to denote the variables for this run of the algorithm; so we have $\mathrm{FLIP}'_i$, $b'_{n'_i}$, $w'_i$, $v'_i$, $c'$ and so on. To show $\tau(R',B') = (R,B)$ it suffices to show~$\mathrm{FLIP}'_{c'} = \mathrm{FLIP}_c$. We claim that in fact $c = c'$ and $\mathrm{FLIP}'_{i} = \mathrm{FLIP}_{i}$ for all $0 \leq i \leq c$. Note first of all that~$\mathrm{FLIP}'_{0} = \mathrm{FLIP}_{0}$ because both of these equal the minimum $(k-1)$-antichain of $\mathrm{Int}^{\Pi}$ and we gave a characterization earlier of this antichain just in terms of~$\mathrm{Int}^{\Pi}$ as an abstract poset, independent of how it is labeled. It is clear that~$m'_0 = m_0$. If~$c = 0$, then no vertex on $b_{m_0}$ ever flips, so no vertex in this path belongs to the downset of~$U$ and thus we get~$c' = 0$ as well. If~$c > 0$ we get $w_1 = w'_1$ as these are both equal to the first place where $b_{m_0}$ intersects the downset of $U$. But then if $w_i = w'_i$, we get $v'_i = v_i$. This is because each element of $\mathrm{Int}^{\Pi'}$ has two paths in $\Pi'$ coming into it, and since~$R'$ and $B'$ are noncrossing these paths must be colored differently (where by the color of the path we mean in the sense of Figure~\ref{fig:tswapex}). We followed one of these paths in to arrive at $w'_i$ and thus we must follow the other out to arrive at $v'_i$. And if~$v'_i = v_i$ and $i < c$, then similarly we have~ $w'_{i+1} = w_{i+1}$. If~$i = c$ then both algorithms terminate on this step and so $c = c'$. The result follows by induction. \end{proof}

\begin{claim} \label{claim:swap}
There exists $J' \in \binom{[2k-1]}{k-1}$ which is a swap of $J$ so that for all paths~$\pi' \in \Pi'$ there is some $j' \in J'$ with $t_{j'}$ an end point of $\pi'$.
\end{claim}
\begin{proof} Let $j^*$ be such that $t_{j^*}$ is the endpoint of $b_{n_c}$. Then we may satisfy the claim by taking $J' := [2k-1] \setminus (J\cup\{j^*\})$. Indeed, for each $r \in R$, there are an odd number of $v \in \mathrm{FLIP}_c$ with $r \in \ell^{\Pi}(v)$; this is easily seen by considering the~$\mathrm{FLIP}_i$ inductively. Thus, as we follow the path $r$ after having flipped at all the vertices in $ \mathrm{FLIP}_c$, the color of the path will be red for a while, then blue, then red, and so on, and must eventually end blue. Also, there are an even number of $v \in \mathrm{FLIP}_c$ with $b_{n_c} \in \ell^{\Pi}(v)$: we can again easily prove inductively that for each $0 \leq i \leq c$ the number of $v \in \mathrm{FLIP}_i$ with $b_{n_i} \in \ell^{\Pi}(v)$ is even, while the number of $v \in \mathrm{FLIP}_i$ with $b \in \ell^{\Pi}(v)$ is odd for any $b \neq b_{n_i}$. Thus the end of~$b_{n_c}$ remains blue. So the endpoints of paths in $R'$ must be among $t_{j'}$ for~$j' \in J'$ as claimed. \end{proof}

Let $D$ denote the disc into which $G$ is embedded, and let $\partial D$ denote its boundary. Let $b \in B$ with start point $s$ and end point $t$. Denote by $\mathrm{rt}(b)$ (respectively, $\mathrm{lt}(b)$) the compact subset of the plane whose boundary is the closed curve obtained by adjoining $b$ with the arc on $\partial D$ that connects $t$ to~$s$ clockwise (respectively, counter-clockwise). It is easy to see~$\mathrm{rt}(b),\mathrm{lt}(b) \subseteq D$ and $\mathrm{rt}(b) \cap \mathrm{lt}(b) = b$. Also, we have $b_j \in \mathrm{rt}(b_i)$ if and only if~$j \geq i$, and similarly~$b_j \in \mathrm{lt}(b_i)$ if and only if~$j \leq i$. For~$r \in R$ we define $\mathrm{rt}(r)$ and~$\mathrm{lt}(r)$ analogously, and have the similar result that $r_j \in \mathrm{rt}(r_i)$ if and only if $j \geq i$, and~$r_j \in \mathrm{lt}(r_i)$ if and only if $j \leq i$.

\begin{lemma} \label{lemma:achainpair}
For $x_1, x_2 \in \mathrm{Int}^{\Pi}$ which are not related, if $\ell^{\Pi}(x_1) = \{r_{i_1},b_{j_1}\}$ and~$\ell^{\Pi}(x_2) = \{r_{i_2},b_{j_2}\}$ then $i_1 \leq i_2$ if and only if~$j_1 \leq j_2$.
\end{lemma}
\begin{proof} We may assume the inequalities of the indices are strict because otherwise~$x_1$ and~$x_2$ would certainly be related. So let $x_1, x_2 \in \mathrm{Int}^{\Pi}$ be such that~$\ell^{\Pi}(x_1) = \{r_{i_1},b_{j_1}\}$ and $\ell^{\Pi}(x_2) = \{r_{i_2},b_{j_2}\}$ where $i_1 < i_2$ but~$j_1 > j_2$. Let~$s_{p_1}$ be the start point of $r_{i_1}$ and $t_{q_1}$ its end point, and let $s_{p_2}$ be the start point of $b_{j_1}$ and $t_{q_2}$ its end point. Assume by symmetry that $p_1 \leq p_2$, so $q_1 \leq q_2$. Let $Y_1$ (respectively, $Y_2$)  denote the compact subset of the plane that is bounded by the closed curve obtained by adjoining the subpath of~$r_{i_1}$ connecting $s_{p_1}$ to~$x_1$ (resp., the subpath of the reverse of $r_{i_1}$ connecting~$t_{q_1}$ to~$x_1$), the subpath  of the reverse of $b_{j_1}$ connecting $x_1$ to $s_{p_2}$ (resp, the subpath of $b_{j_1}$ connecting $x_1$ to $t_{q_2}$), and the arc on $\partial D$ connecting $s_{p_2}$ to~$s_{p_1}$ counter-clockwise (resp., the arc on $\partial D$ connecting $t_{q_2}$ to $t_{q_1}$ clockwise). Because $x_2$ lies in $\mathrm{lt}(r_{i_1}) \cap \mathrm{rt}(b_{j_1})$, it must lie in one of $Y_1$ or $Y_2$. Assume by symmetry that it lies in $Y_1$. We claim that the subpath of $r_{i_2}$ below $x_2$ cannot lie inside $Y_1$: if it did, its end point would lie clockwise between $s_{p_1}$ and $s_{p_2}$ on $\partial D$, contradicting our assumption about how sources and sinks of $G$ are arranged on this boundary. So it must exit $Y_1$. When it does so, it crosses $b_{j_1}$ above $x_1$. Thus $x_2 > x_1$. \end{proof}

\begin{lemma} \label{lemma:indexseq}
For $r \in R$, if $x_1 < \cdots < x_l$ are the elements of~$\mathrm{Int}^{\Pi} \cap \mathrm{Vert}(r)$ and~$b_{p_i} \in \ell^{\Pi}(x_i)$ for all $i \in [l]$, then $|p_i - p_{i-1}| \leq 1$ for all~$i > 1$.
\end{lemma}
\begin{proof} This is an immediate consequence of the facts that $G$ is planar and~$B$ is noncrossing. \end{proof}

\begin{claim} \label{claim:balancedswap}
The sink pattern $J'$ is a balanced swap of $J$.
\end{claim}
\begin{proof} Claim~\ref{claim:swap} tells us that $J'$ and $J$ are swaps of one another and their pivot~$j^*$ is such that $t_{j^*}$ is the endpoint of $b_{n_c}$. Define $w_{c+1}$ to be $t_{j^*}$. Define~$v_0$ to be first element that comes strictly before $w_1$ in $b_{n_0}$ and belongs to $\mathrm{Int}^{\Pi}$, or to be the start point of $b_{n_0}$ if there is no such element. Then for~$0 \leq i \leq c$, let~$\widetilde{b}_{i}$ be the subpath of $b_{n_i}$ connecting $v_i$ to $w_{i+1}$. Also, for~$i \in [k]$ define closed subsets~$X_i$ of $D$ by $X_1 := \mathrm{rt}(r_1)$, $X_i := \mathrm{lt}(r_{i-1}) \cap \mathrm{rt}(r_{i})$ if $1 < i < k$, and~$X_k := \mathrm{lt}(r_{k-1})$.

Our key subclaim is that for $0 \leq i \leq c$, the curve $\widetilde{b}_{i}$ lies in $X_{n_i}$. We prove this by induction on $i$. First of all, for any $i$ it is clear that each $\widetilde{b}_{i}$ must lie in one of the $X_j$ because if it did not it would have to intersect too many paths in~$R$ (only the start point and end point of  $\widetilde{b}_{i}$ can belong to~$\mathrm{Int}^{\Pi}$). So each~$\widetilde{b}_{i}$  intersects the interior of at most one of the $X_j$. First we deal with the base case~$i = 0$. Assume that $n_0 \in [2,2k-2]$; the cases~$n_0 = 1$ or $n_0 = 2k-1$ are very similar to what follows. By Lemma~\ref{lemma:achainpair}, there exist~$u_1, u_2 \in U$ with~$\ell^{\Pi}(u_1) = \{b_{n_0-1},r_{n_0-1}\}$ and $\ell^{\Pi}(u_2) = \{b_{n_0}+1,r_{n_0}\}$. Let $s_{j_1}$ be the start point of $b_{n_0-1}$, $s_{j_2}$ the start point of $b_{n_0+1}$, $s_{j_3}$ the start point of~$r_{n_0-1}$, and~$s_{j_4}$ the start point of $r_{n_0}$. Assume that $j_1 \leq j_3$ and $j_2 \leq j_4$ for simplicity of the following exposition; the other cases are symmetric. Let $Y_1$ (respectively,~$Y_2$) denote the compact subset of the plane that is bounded by the closed curve obtained by adjoining the subpath of~$b_{n_0-1}$ (resp., $b_{n_0 + 1}$) connecting $s_{j_1}$ to $u_1$ (resp., $s_{j_2}$ to $u_2$), the subpath of the reverse of~$r_{n_0-1}$ (resp., $r_{n_0}$) connecting $u_1$ to $s_{j_3}$ (resp., $u_2$ to $s_{j_4}$), and the arc on $\partial D$ connecting $s_{j_3}$ to~$s_{j_1}$ (resp., $s_{j_4}$ to $s_{j_2}$) counter-clockwise. We claim that $\widetilde{b}_0$ cannot intersect the interior of $Y_1$ or of $Y_2$. Suppose to the contrary; by symmetry, assume that $\widetilde{b}_0$ enters the interior of $Y_1$. First of all, if~$c=0$, that means the end point of $b_{n_0}$ lies on $\partial D$ clockwise between~$s_{j_1}$ and~$s_{j_3}$, but because~$j_1 < j_3$, this contradicts our assumption of how the sources and sinks of $G$ are arranged on this boundary. (Note $j_1 = j_3$ is impossible in this case because that would force the end point of $b_{n_0}$ to be~$s_{j_1}$ as well, creating a cycle.) So suppose $c > 0$. Then $w_1$ is in $Y_1$. It cannot be on $r_{n_0-1}$ because that would put it at or above $u_1$. But because the path~$b_{n_0}$ must exit $Y_1$, it must cross $r_{n_0-1}$ above $u_1$ after it visits $w_1$, and this means that~$u_1 < w_1$. But because $w_1$ belongs to the downset of $U$, this contradicts the fact that $U$ is an antichain. So indeed $\widetilde{b}_0$ does not enter the interior of $Y_1$ or of $Y_2$. But $b_{n_0}$ lies in $\mathrm{lt}(b_{n_0-1}) \cap \mathrm{rt}(b_{n_0+1})$; therefore~$\widetilde{b}_{n_0}$ lies in $\mathrm{lt}(r_{n_0-1}) \cap \mathrm{rt}(r_{n_0}) = X_{n_0}$.

Now assume $i > 0$ and the key subclaim holds for smaller values of $i$. We know that $\widetilde{b}_{i-1}$ lies in $X_{n_{i-1}}$, so either $r_{n_{i-1}-1}$ or $r_{n_{i-1}}$ is in $\ell^{\Pi}(w_i)$; let us assume by symmetry that it is $r_{n_{i-1}}$. Lemma~\ref{lemma:indexseq} gives $n_i = n_{i-1} + \delta$ for~$\delta \in \{-1,0,1\}$. First suppose that $\delta = 0$. Then we claim that $\widetilde{b}_{i}$ cannot enter the interior of~$X_{n_i+1}$; in particular, the subpath of $b_{n_i}$ connecting $v_i$ to $w_i$ cannot enter the interior of~$X_{n_i+1}$. Suppose that it did. This subpath must eventually enter the interior of $X_{n_i}$ by our inductive supposition and so it would have to cross $r_{n_i}$ at some point to do so. However, if it crossed~$r_{n_i}$ above $v_i$ this would cause a cycle in $G$, and if it crossed below $w_i$ this would also cause a cycle. So it would have to cross between $v_i$ and $w_i$; but this is also impossible because there are no elements of $\mathrm{Int}^{\Pi}$ that lie on $r_{n_i}$ between~$v_i$ and $w_i$. So indeed $\widetilde{b}_{i}$ lies in $X_{n_i}$.

Now suppose that $\delta \neq 0$. Let $s_{j_1}$ be the start point of~$r_{n_{i-1}}$ and $s_{j_2}$ be the start point of $b_{n_{i-1}}$. Assume $j_1 \leq j_2$ for simplicity of the following exposition; the other case is symmetric. Let $Y$ denote the compact subset of the plane that is bounded by the closed curve obtained by adjoining the subpath of~$b_{n_{i-1}}$ connecting $s_{j_1}$ to $w_i$, the subpath of the reverse of~$r_{n_{i-1}}$ connecting $w_i$ to $s_{j_2}$, and the arc on $\partial D$ connecting $s_{j_2}$ to~$s_{j_1}$ counter-clockwise. We claim that the subpath of $b_{n_i}$ below $v_i$ cannot enter the interior of $Y$. Suppose it did. Then it could not exit $Y$ because it cannot intersect $b_{n_{i-1}}$ at all, and it cannot intersect~$r_{n_{i-1}}$ above $w_i$ without creating a cycle. Thus the end point of $b_{n_i}$ lies on $\partial D$ clockwise between $s_{j_1}$ and $s_{j_2}$. But because $j_1 < j_2$, this contradicts our assumption of how the sources and sinks of $G$ are arranged on this boundary. (Note $j_1 = j_2$ is impossible in this case because that would force the end point of $b_{n_i}$ to be $s_{j_1}$ as well, creating a cycle.) So $\widetilde{b}_{i}$ does not enter the interior of $Y$. Thus $\widetilde{b}_{i}$ lies in $X_{n_{i-1}+1}$.

To finish the proof of the key subclaim, we need to show that $\delta \neq -1$. First consider the case $i = 1$. Let $u$ be the unique element of $U$ such that~$r_{n_0} \in \ell^{\Pi}(u)$. Note that $b_{n_0+1} \in \ell^{\Pi}(u)$ by Lemma~\ref{lemma:achainpair}. Also note that $w_1$ is the maximal element below $u$ with $\ell^{\Pi}(w_1) = \{r_{n_0},b_{n_0}\}$. So by Lemma~\ref{lemma:indexseq}, as we look at the~$b_{p_j}$ intersecting the vertices of~$r_{n_0}$ we encounter below $u$ but above $w_1$ we can never see $b_{n_0-1}$. Thus $\delta = 1$ as claimed. Now consider the case $i > 1$. Then either $r_{n_{i-1}-1}$ or $r_{n_{i-1}}$ is in $\ell^{\Pi}(v_{i-1})$. Suppose first that~$r_{n_{i-1}-1} \in \ell^{\Pi}(v_{i-1})$. Then $v_{i-1}$ and $v_i$ are unrelated and so by Lemma~\ref{lemma:achainpair} we get that $\delta \neq -1$. Suppose next that $r_{n_{i-1}} \in \ell^{\Pi}(v_{i-1})$. Then note that~$w_i$ is the maximal element below $v_{i-1}$ with $\ell^{\Pi}(w_1) = \{r_{n_{i-1}},b_{n_{i-1}}\}$. Also, the element of $\mathrm{Int}^{\Pi}$ below $v_{i-1}$ on $r_{n_{i-1}}$ is $w_{i-1}$ and has $b_{n_{i-2}} \in w_{i-1}$, so again by Lemma~\ref{lemma:indexseq} we get $b_{n_{i-2}} \in v_i$. So $n_i = n_{i-2}$ and by induction we obtain~$\delta = 1$ again. The key subclaim is thus proved by induction.

To conclude, note that $\widetilde{b}_{n_c}$ is in $X_{n_c}$ which means $t_{j^*}$ is clockwise between the end point of $r_{n_c}$ and $r_{n_c+1}$ on $\partial D$. Together with Claim~\ref{claim:swap}, this means exactly that $|J \cap [j^*]| = |J' \cap [j^*]|$. \end{proof}

\begin{claim} \label{claim:endswap}
If $G$ is interlacing then $J'$ is an end swap of $J$.
\end{claim}
\begin{proof} Let $\Pi_T$ be the subtuple of $\Pi$ consisting of paths whose end points are among $t_j$ for $j \in [2,2k-2]$. Let $N_T$ be the subset of $V$ guaranteed by the $(k-1)$-sink-bottlenecked property of $G$. There is a subset of $N_T$ of size~$k-2$ consisting of $2$-crossings of $\Pi_T$. This subset is a $(k-2)$ antichain of~$\mathrm{Int}^{\Pi_T}$ because~$N_T$ is non-returning. There are no antichains of $\mathrm{Int}^{\Pi_T}$ of greater cardinality. So~$\mathcal{A}_{k-2}(\mathrm{Int}^{\Pi_T})$ has a minimum; call that minimum $U_T$. We claim that $U_T$ belongs to the downset of $U$. To see this, let $U^T  \subseteq U$ be the set of those $u \in U$ for which $\ell^{\Pi}(u) \subseteq \Pi_T$. Note that because $N_T$ is sink-branching, it also must be that $U_T$ is sink-branching. So no element of~$U_T$ is greater than an element of $U \setminus U^T$; but also, every element of $U_T$ is comparable to some element of $U$. Thus if we let $U_{\mathrm{min}}$ be the set of minimal elements of $U \cup U_T$, there is a subset of $U_{\mathrm{min}}$ that belongs to~$\mathcal{A}_{k-2}(\mathrm{Int}^{\Pi_T})$ and is in the downset of~$U$. But~$U_T$ is minimal among all such antichains; so~$U_T$ must be in the downset of~$U$.

 Let $j^*$ be the pivot of $J$ and $J'$. We want to show that $j^* \notin [2,2k-2]$. Suppose to the contrary.  Recall the paths $\widetilde{b}_{i}$ and regions $X_{i}$ defined in the proof of Claim~\ref{claim:balancedswap}.  Let $\pi^{*}$ be the unique element of $\Pi_T$ not among the labels of elements of $U_T$. If $j^* \in [2,2k-2]$, it must be that there is $i$ such that~$\widetilde{b}_{i} = \pi^{*}$ and either $w_{i+1}$ is in he downset of $U_T$ or $w_{i+1} = t_{j*}$. This is because if $\widetilde{b}_{i}$ passes through some $u \in U_T$, that $u$ must either be $v_i$ or $w_{i+1}$. If that~$u$ is a~$w_{i+1}$, then~$v_{i+1}$ will not belong to the downset of $U_T$ so we will have to pass through~$U_T$ again at some later step. On the other hand, if that $u$ is a $v_{i}$, then we must have $i > 0$ and $w_{i-1}$ already belongs to the downset of $U_T$ and is strictly below an element of $U_T$. Thus indeed there exists $i$ such that $\widetilde{b}_{i}= \pi^{*}$ and with $w_{i+1}$ as described above. But then by the same logic as the second paragraph of the proof of Claim~\ref{claim:balancedswap}, we conclude that either $\widetilde{b}_i$ lies in $X_{n_i+1}$ (if~$1 \in J$) or $\widetilde{b}_i$ lies in $X_{n_i-1}$ (if $1 \notin J$). At any rate, we get that $\widetilde{b}_i$ does not lie in~$X_{n_i}$ which is a contradiction with the key subclaim in the proof of Claim~\ref{claim:balancedswap}. So indeed $j^* \notin [2,2k-2]$. \end{proof}
 
\begin{proof}[Proof of Theorem~\ref{thm:tswap}]: Claims~\ref{claim:algwelldef} and~\ref{claim:algterminates} establish that $\tau$ is well-defined, and Claim~\ref{claim:noncrossing} shows that $\tau$ maps into $\mathrm{PNCPath}(G)$. The map $\tau$ is weight-preserving because the multisets of edges visited by paths in $\Pi$ and in~$\Pi'$ are identical. Claim~\ref{claim:involution} shows $\tau$ is an involution. Claim~\ref{claim:balancedswap} describes the image~$\tau(\mathrm{PNCPath}(I,J))$, and Claim~\ref{claim:endswap} gives a more refined estimate on~$\tau(\mathrm{PNCPath}(I,J))$ when $G$ is interlacing. \end{proof}

\begin{proof}[Proof of Corollary~\ref{cor:coroftswap1}]:  For any $J, J' \in \binom{[2k-1]}{k-1}$ with $J'$ a balanced swap of~$J$, their pivot $j^*$ cannot be even, so we have $J'_{\mathrm{even}} = [2k-1]_{\mathrm{even}} \setminus J_{\mathrm{even}}$. Thus with~$K$ as in the statement of the corollary, by Theorem~\ref{thm:tswap} we have
\[ \tau \bigg( \bigcup_{\substack{(I,J) \in \mathrm{Pat}(G) \\ J_{\mathrm{even}} = K}} \mathrm{PNCPath}(I,J) \bigg) \subseteq\bigcup_{\substack{(I,J') \in \mathrm{Pat}(G) \\ J'_{\mathrm{even}}= K'}} \mathrm{PNCPath}(I,J').\]
But the reverse inclusion follows for the same reason. \end{proof}

\begin{proof}[Proof of Corollary~\ref{cor:coroftswap2}]: For $J, J', J'' \in \binom{[2k-1]}{k-1}$ as in the statement of the corollary, we have $\mathrm{eswap}(J) \subseteq J' \cup J''$, but we also have~$\mathrm{eswap}(J') \subseteq J$ and~\mbox{$\mathrm{eswap}(J'') \subseteq J$}. Then Theorem~\ref{thm:tswap} tells us that
\[\tau(\mathrm{PNCPath}(I,J)) \subseteq \mathrm{PNCPath}(I,J') \cup \mathrm{PNCPath}(I,J'')\]
and also the reverse inclusion. \end{proof}

\begin{remark} \label{rem:sourceswap}
The definition of $k$-bottlenecked is symmetric with respect to sources and sinks, so we can easily obtain from $\tau$ a source-swapping involution as well. We define $(G^{\mathrm{op}},S^{\mathrm{op}},T^{\mathrm{op}})$, the opposite network of $G$, as follows: $G^{\mathrm{op}}$ is the same graph as $G$ but with edge directions reversed, $S^{\mathrm{op}} := (t_{2k-1},\ldots,t_1)$, and $T^{\mathrm{op}} := (s_{2k-1},\ldots,s_1)$. For $I \subseteq [2k-1]$ set~$I^{\circ} := \{ 2k - i\colon i \in I\}$. There is a weight-preserving bijection~$\Psi\colon \mathrm{PNCPath}(G) \to \mathrm{PNCPath}(G^{\mathrm{op}})$ such that~$\Psi(\mathrm{PNCPath}(I,J)) = \mathrm{PNCPath}(I^{\circ},J^{\circ})$ for $(I,J) \in \mathrm{Pat}(G)$ whereby~$\Psi$ just reverses all paths. Suppose $G$ is $k$-bottlenecked. Then so is $G^{\mathrm{op}}$. So we may define the source-swapping involution $\sigma\colon \mathrm{PNCPath}(G) \to \mathrm{PNCPath}(G)$ by~$\sigma := \Psi^{-1} \circ \tau_{G^{\mathrm{op}}} \circ \Psi$ and it will satisfy
\[\sigma(\mathrm{PNCPath}(I,J)) \subseteq \bigcup_{I' \in \mathrm{bswap}(I)} \mathrm{PNCPath}(I',J),\]
for all $(I,J) \in \mathrm{Pat}(G)$. Here $\tau_{G^{\mathrm{op}}}$ denotes the involution $\tau$ defined above in this section but applied to the opposite network. The involution $\sigma$ leads to a source-swapping analogue of Corollary~\ref{cor:coroftswap1}. However, $G$ being interlacing does not in general imply that $G^{\mathrm{op}}$ is interlacing, so we do not in general get a source-swapping analogue of Corollary~\ref{cor:coroftswap2}.
\end{remark}

\section{Interlacing matrices and Pl\"{u}cker relations} \label{sec:matrices}

In this section, we provide an alternative algebraic proof to the second part of Corollary~\ref{cor:coroftswap2} using certain rank properties of matrices associated to interlacing networks. Here and throughout all matrices are real. A matrix is called totally nonnegative if all its minors are nonnegative. We define~$\mathrm{Mat}(m,n)$ to be the set of~$m \times n$ matrices and $\mathrm{Mat}^{\geq0}(m,n)$ the set of~$m \times n$ totally nonnegative matrices. We also use~$\mathrm{Mat}_{*}(m,n)$ to denote the set of~$m \times n$ matrices of full rank, and~$\mathrm{Mat}^{\geq0}_{*}(m,n)$ for the totally nonnegative matrices of full rank. The Lindstr\"{o}m--Gessel--Viennot (LGV)~\cite{lindstrom}~\cite{gessel} lemma provides a correspondence between totally nonnegative matrices and planar networks, and we would like to examine the properties of the class of totally nonnegative matrices that correspond to interlacing networks. Let~$(G,S,T)$ be a network. Define~$P_G$ to be the matrix whose entry at~$i,j$ is equal to~$\sum_{\pi \in \mathrm{NCPath}(\{s_i\},\{t_j\})}\mathrm{wt}(\pi)$. By the LGV lemma, since~$G$ is a planar, the matrix~$P_G$ is totally nonnegative. Moreover, we have $\det P_G[U|W]=\sum_{\Pi \in \mathrm{NCPath}(U,W)}\mathrm{wt}(\Pi)$ for $ m \in [2k-1]$ and~$U,W \in \binom{[2k-1]}{m}$

\begin{definition}
Let $k \geq 2$. An \emph{interlacing matrix} of order $2k-1$ is a totally nonnegative $2k-1 \times 2k-1$ matrix $A$ whose rank is at most $k$, such that the rank of $A\big[[2k-1]|[2,2k-2]\big]$ is at most $k-1$.
\end{definition}

\begin{prop} \label{prop:interlacing matrix}
Let $(G,S,T)$ be an interlacing network. Then $P_G$ is interlacing.
\end{prop}
\begin{proof} In order to prove that the rank of $P_G$ is at most $k$, it is enough to show that every $(k+1) \times (k+1)$ minor of $P_G$ equals 0. Let $U,W \in \binom{[2k-1]}{k-1}$. Since~$G$ is $k$-bottlenecked, $\mathrm{NCPath}(U,W)=\emptyset$ and hence by the LGV lemma we have $\det P_G[U|W]=\sum_{\Pi \in \mathrm{NCPath}(U,W)}\mathrm{wt}(\Pi)=0$. Similarly, since $G$ is~$(k-1)$-sink-bottlenecked, after removing the first and the last column we get a matrix in which every $k \times k$ minor equals 0. \end{proof}

We therefore can reformulate the second part of Corollary~\ref{cor:coroftswap2} in terms of interlacing matrices as follows.

\begin{thm}\label{thm:intermat}
Let $k \geq 2$ and let $M$ be an interlacing matrix of order $2k-1$. Fix $I, J \in \binom{[2k-1]}{k-1}$ such that~$\{1,2k-1\} \cap J = \emptyset$. Set $J' := [2,2k-1] \setminus J$ and~$J'' = [1,2k-2] \setminus J$. Then
\[ \det M[\overline{I}|\overline{J}] \det M[I|J]=\det M[\overline{I}|\overline{J'}] \det M[I|J']+\det M[\overline{I}|\overline{J''}] \det M[I|J''] \]
where for $K \subseteq [2k-1]$ we define $\overline{K}=[2k-1]\setminus K$.
\end{thm}

The proof of Theorem~\ref{thm:intermat} relies on the Pl\"{u}cker relations between minors of certain types of Grassmannians. For $n\geq l\geq 0$, the \emph{Grassmannian} $\mathrm{Gr}(l,n)$ is the space of~$l$-dimensional linear subspaces of $\mathbb{R}^n$. Another way to view the Grassmannian is as $\mathrm{Gr}(l,n) = GL(l) \setminus \mathrm{Mat}_{*}(l,n)$, where $GL(l)$ is the group of invertible~\mbox{$l \times l$} matrices. In other words, we can identify $\mathrm{Gr}(l,n)$ with the space of $l \times n$ real matrices of rank $l$ modulo row operations, where the~\mbox{$l \times l$} minors of the matrices form projective coordinates on the Grassmannian, called Pl\"{u}cker coordinates.  We denote those coordinates by~$\Delta_I(A)$ for~\mbox{$I \in \binom{[n]}{l}$} and~\mbox{$A \in \mathrm{Gr}(l,n)$}; that is,~$\Delta_I(A)$ represents the minor of $A$ defined by the columns $I=\{i_1,i_2\ldots,i_l\}$ with~\mbox{$i_1<i_2<\ldots<i_l$}. When~$A$ is clear from context we write simply $\Delta_I$ for~$\Delta_I(A)$. In order to simplify notation when swapping columns of minors, we use the conventions on these coordinates that~$\Delta_{(i_1,i_2,\ldots,i_l)}:=\Delta_{\{i_1,i_2\ldots,i_l\}}$ for~\mbox{$i_1<\ldots<i_l$} and~$\Delta_{(i_1,\ldots,i_j,i_{j+1},\ldots,i_l)}:=-\Delta_{(i_1,\ldots,i_{j+1},i_{j},\ldots,i_l)}$. The following set of relations on Pl\"{u}cker coordinates for any choice of~$m \in [l]$ are called the Pl\"{u}cker relations~\cite{postnikov}:
\small
\begin{equation}\label{eqn:pluckerrelations}
\Delta_{(p_1,\ldots,p_l)}\Delta_{(q_1,\ldots,q_l)}= \sum_{i_1< \ldots < i_m}\Delta_{(p_1,\ldots, q_{l-m+1},\ldots,q_l,\ldots p_l)}\Delta_{(q_1,q_2,\ldots,q_{l-m},p_{i_1},\ldots,p_{i_m})}.
\end{equation}
\normalsize
Here $(p_1,\ldots, q_{l-m+1},\ldots,q_l,\ldots p_l)$ denotes the tuple $(p_1,\ldots,p_l)$ with the entries~$p_{i_1},\ldots,p_{i_m}$ replaced by $q_{l-m+1},\ldots,q_l$ and vice versa for the other factor.

In a manner analogous to the above description of the Grassmannian as a quotient of a matrix space by a general linear group action, we define the \emph{totally nonnegative Grassmannian} to be $\mathrm{Gr}^{\geq 0}(l,n) :=  GL^{+}(l) \setminus \mathrm{Mat}_{*}^{\geq 0}(l,n)$, where $GL^{+}(l)$ is the group of invertible $l \times l$ matrices with positive determinant. In other words, the totally nonnegative Grassmannian is the subset of the Grassmannian for which Pl\"{u}cker coordinates are all nonnegative (or rather all of the same sign, since these coordinates are projective). So~$\mathrm{Gr}^{\geq 0}(l,n) \subseteq \mathrm{Gr}(l,n)$. There is a tight correspondence between totally nonnegative matrices and the totally nonnegative Grassmannian. In fact, there exists an embedding~$\phi\colon \mathrm{Mat}^{\geq 0}(l,n) \to \mathrm{Gr}^{\geq0}(l,l+n)$ of the form
\[\phi\colon A  \mapsto \begin{pmatrix}
1 & 0 & \cdots & 0 & 0 & 0 & (-1)^{l-1} a_{l1} & (-1)^{l-1} a_{l2} & \cdots & (-1)^{l-1} a_{ln} \\
\vdots & \vdots & \ddots & \vdots & \vdots & \vdots & \vdots & \vdots & \ddots & \vdots\\
0 & 0 & \cdots & 1 & 0 & 0 & a_{31} & a_{32} & \cdots & a_{3n}\\
0 & 0 & \cdots & 0 &1 & 0 & -a_{21} & - a_{22} & \cdots & -a_{2n} \\
0 & 0 & \cdots & 0 & 0 & 1 & a_{11} & a_{12} & \cdots & a_{1n}
\end{pmatrix}\]
such that
\begin{equation}\label{eq:mattograssman}
\det A[I|J] = \Delta_{([l]\setminus\{l+1-i_r,\dots,l+1-i_1\})\cup \{j_1+l,\dots,j_r+l\}}(\phi(A))
\end{equation}
for~$A=(a_{ij}) \in \mathrm{Mat}^{\geq 0}(l,n)$, $I=\{i_1,\dots,i_r\} \subseteq [l]$ and $J=\{j_1,\dots,j_r\} \subseteq [n]$.

For a number $a$ and a set $B=\{b_1,b_2,\ldots,b_r\}$ such that $b_1<b_2<\ldots<b_r$, denote by $\{a+B\}$ (respectively, $\{a-B\}$) the set $\{a+b_1,a+b_2,\ldots,a+b_r\}$ (resp., $\{a-b_r,a-b_{r-1},\ldots,a-b_1\}$). Let us set $l = 2k-1$ and $n = 2k-1$ in the last paragraph to get $\phi\colon  \mathrm{Mat}^{\geq 0}(2k-1,2k-1) \to \mathrm{Gr}^{\geq0}(2k-1,4k-2)$. Then the conclusion of Theorem~\ref{thm:intermat} is equivalent to the following equation on the Pl\"{u}cker coordinates of $\phi(M)$:
\begin{align}\label{eqn:pluckereq}
&\Delta_{[2k-1]\setminus \{2k - \overline{I}\} \cup \{2k-1+\overline{J}\}}\Delta_{[2k-1]\setminus \{2k - I\} \cup \{2k-1+J\}}=  \\
    &\Delta_{[2k-1]\setminus \{2k - \overline{I}\} \cup \nonumber \{2k-1+\overline{J'}\}}\Delta_{[2k-1]\setminus \{2k - I\} \cup \{2k-1+J'\}}\\ \nonumber
   &+\Delta_{[2k-1]\setminus \{2k - \overline{I}\} \cup  \{2k-1+\overline{J''}\}}\Delta_{[2k-1]\setminus \{2k - I\} \cup \{2k-1+J''\}}. \nonumber
\end{align}
We are now ready to present the proof of Theorem~\ref{thm:intermat}.

\noindent {\bf Proof of Theorem~\ref{thm:intermat}}: Because $M$ is interlacing, we have $\det M[W|V]=0$ if~$|W|=|V|>k$ or if $|W|=|V|>k-1$ and $1,2k-1 \notin V$.
Using (\ref{eq:mattograssman}), we get that the Pl\"{u}cker coordinates of $\phi(M)$ satisfy
\begin{enumerate}[(a)]
\item $\Delta_U=0$ for any $U=\{u_1,u_2,\dots,u_{2k-1}\}$ such that $u_1<\dots<u_{2k-1} $ and~$ u_{k-1}>2k-1 $; \label{cond:pluckvanish}
\item $\Delta_U=0$ for any $U=\{u_1,u_2,\dots,u_{2k-1}\}$ such that $u_1<\dots<u_{2k-1} $, $2k,2(2k-1) \notin U$ and $u_k>2k-1$. \label{cond:pluckvanish2}
\end{enumerate}

Now, note that since $\overline{J'}=\{1\} \cup J$ and $\overline{J''}=J \cup \{2k-1\}$, equation~\eqref{eqn:pluckereq} is equivalent to
\begin{align}\label{eqn:pluckereqequiv}
 &\Delta_{\{2k-I\} \cup [2k,2(2k-1)]\setminus \{2k - 1+ J\}}\Delta_{[2k-1]\setminus \{2k - I\} \cup \{2k-1+J\}} = \\ \nonumber
 & \Delta_{\{2k-I\} \cup \{2k\} \cup \{2k-1+J\}}\Delta_{[2k-1]\setminus \{2k - I\} \cup \{2k-1+J'\}}\\ \nonumber
 &+\Delta_{\{2k-I\} \cup  \{2k-1+ J\} \cup \{2(2k-1)\}}\Delta_{[2k-1]\setminus \{2k - I\} \cup \{2k-1+J''\}}. \nonumber
\end{align}
To show that~\eqref{eqn:pluckereqequiv} holds, we will use~\eqref{eqn:pluckerrelations} with $l=2k-1$ and $m=k-1$. According to the formula, we are summing over all the $\binom{2k-1}{k-1}$ ways in which we can put the $k-1$ elements of $\{2k-1+J\}$ in place of some~$k-1$ elements of~\mbox{$\{2k-I\} \cup [2k,2(2k-1)]\setminus \{2k - 1+ J\}$}. We will show that only two summands among the $\binom{2k-1}{k-1}$ summands appear on the right side of~\eqref{eqn:pluckerrelations} may be nonzero, and they are equal to the right side of~\eqref{eqn:pluckereqequiv}. First, consider the summands in which at least one element from~\mbox{$\{2k-1+J\}$} is placed instead of an element in $\{2k-I\}$. Since all of the elements in~\mbox{$[2k,2(2k-1)]\setminus \{2k - 1+ J\}$} are bigger than $2k-1$ and~$|\{2k-I\}|=k-1$ we are in the case~\ref{cond:pluckvanish}, which means that the resulting summand equals zero. Thus in order to obtain a nonzero summand, all the $k-1$ elements from~$\{2k-1+J\}$ must be placed instead of some $k-1$ elements from~\mbox{$[2k,2(2k-1)]\setminus \{2k - 1+ J\}$}. There are exactly~\mbox{$\binom{k}{k-1}=k$} such summands since~\mbox{$|[2k,2(2k-1)]\setminus \{2k - 1+ J\}|=k$}, and in each of the summands exactly one element from the set~$[2k,2(2k-1)]\setminus \{2k - 1+ J\}$ is not replaced by an element from $\{2k-1+J\}$, and all the other are replaced.  Note~$2k, 2(2k-1) \in [2k,2(2k-1)]\setminus \{2k - 1+ J\}$ since~$1,2k-1 \notin J$. We may choose one of the following to be the element that is not replaced: $2k$; $2(2k-1)$; or an element from $[2k,2(2k-1)]\setminus \{2k - 1+ J\}$ not equal to $2k$ or $2(2k-1)$. If we choose $2k$, the resulting summand is
\[\Delta_{\{2k-I\} \cup \{2k\} \cup \{2k-1+J\}}\Delta_{[2k-1]\setminus \{2k - I\} \cup \{2k-1+J'\}},\]
 and if we choose $2(2k-1)$ the resulting summand is
\[\Delta_{\{2k-I\} \cup  \{2k-1+ J\} \cup \{2(2k-1)\}}\Delta_{[2k-1]\setminus \{2k - I\} \cup \{2k-1+J''\}}.\]
If we choose an element which is not equal to $2k$ or $2(2k-1)$, then~\ref{cond:pluckvanish2} implies that the $k-2$ resulting summands equal zero. Thus we showed that~\eqref{eqn:pluckereqequiv} holds and so we are done. $ \square$

Because interlacing matrices are defined by certain rank conditions on submatrices, they can also be characterized as those totally nonnegative matrices for which some explicit set of minors is zero. This simple observation connects interlacing matrices to the combinatorial theory of total positivity developed by Postnikov in~\cite{postnikov}, which can be seen as an ``elementary'' approach to the general theory of total positivity initiated by Lusztig~\cite{lusztig}. We very briefly recap the matroid stratification of the totally nonnegative Grassmannian, without even defining exactly what a matroid is. For any~$\mathcal{M} \subseteq \binom{[n]}{l}$, define $\mathcal{S}_{\mathcal{M}} \subseteq \mathrm{Gr}(l,n)$ by
\[\mathcal{S}_{\mathcal{M}} := \{A \in \mathrm{Gr}(l,n)\colon \Delta_I(A) = 0 \textrm{ if and only if } I \in \mathcal{M}\}.\]
Define $\mathcal{S}_{\mathcal{M}}^{\geq 0} := \mathcal{S}_{\mathcal{M}} \cap \mathrm{Gr}^{\geq 0}(l,n)$. The $\mathcal{S}_{\mathcal{M}}^{\geq 0}$ stratify $ \mathrm{Gr}^{\geq 0}(l,n)$ in the sense that~$\cup_{\mathcal{M} \in \binom{[n]}{l}} \mathcal{S}_{\mathcal{M}}^{\geq 0} = \mathrm{Gr}^{\geq 0}(l,n)$ and $\mathcal{S}_{\mathcal{M}}^{\geq 0} \cap \mathcal{S}_{\mathcal{M}'}^{\geq 0} = \emptyset$ if $\mathcal{M} \neq \mathcal{M}'$. This stratification is called the \emph{matroid stratification} of the totally nonnegative Grassmannian. Also it turns out~\cite{postnikov} that each $\mathcal{S}_{\mathcal{M}}$ is either empty or a cell and conjecturally this stratification gives a regular CW decomposition of $\mathrm{Gr}^{\geq 0}(l,n)$. Define $\mathcal{M}^* \subseteq \binom{[4k-2]}{2k-1}$ by
\small
\[\mathcal{M}^*  := \left\{I \in \binom{[4k-2]}{2k-1}\colon \parbox{3in}{\begin{center}$|I\cap[2k,4k-2]| \geq k \textrm{ or }$ \\ $\Big(|I\cap[2k,4k-2]| \geq k-1 \textrm{ and } \{2k,4k-2\}\cap I =\emptyset\Big)$\end{center}} \right\}.  \]
\normalsize
Then the image of the space of interlacing matrices of order $2k-1$ under the map $\phi\colon  \mathrm{Mat}^{\geq 0}(2k-1,2k-1) \to \mathrm{Gr}^{\geq0}(2k-1,4k-2)$ is $\cup_{\mathcal{M}^* \subseteq \mathcal{M}} \mathcal{S}^{\geq 0}_{\mathcal{M}}$.

\begin{prop} \label{prop:interspace}
The space of interlacing matrices of order $2k-1$, viewed inside the Grassmannian via the map $\phi$, forms the closure of a cell in the matroid stratification of $\mathrm{Gr}^{\geq0}(2k-1,4k-2)$.
\end{prop}
\begin{proof} If $\mathcal{M} \subseteq \binom{[4k-2]}{2k-1}$ is such that the cell $\mathcal{S}^{\geq 0}_{\mathcal{M}}$ is nonempty, then the closure of this cell is given by $\mathrm{clo}(\mathcal{S}^{\geq 0}_{\mathcal{M}}) = \cup_{\mathcal{M} \subseteq \mathcal{M}'} \mathcal{S}^{\geq 0}_{\mathcal{M}'}$. Thus in order to prove the proposition we need only show that $\mathcal{S}^{\geq 0}_{\mathcal{M}^*}$ is nonempty. To do so, we construct an interlacing network $G = (G,S,T)$ such that $\phi(P_G) \in \mathcal{S}^{\geq 0}_{\mathcal{M}^*}$. We use two auxiliary networks $G' = (G',S',T')$ and $G'' = (G'',S'',T'')$ to build $G$. The network $G'$ has underlying graph $\Gamma_{k,4k-3}$ with
\begin{align*}
S' &= (s'_1,\ldots,s'_{2k-1}) := ( (1,1), (1,2),\ldots, (1,2k-1) ); \\
T' &= (t'_1,\ldots,t'_{2k-1}) := (  (k,2k-1), (k,2k),\ldots, (k,4k-3) ).
\end{align*}
Observe that $G'$ is $k$-bottlenecked: we may take
\[N' := ( (1,2k-1), (2,2k-1), \ldots, (k,2k-1) )\]
as our bottleneck. The network $G''$ has underlying graph $\Gamma_{k-1,4k-7}$ with
\begin{align*}
S'' &= (s''_1,\ldots,s''_{2k-3}) := ( (1,1), (1,2),\ldots, (1,2k-3) ); \\
T'' &= (t''_1,\ldots,t''_{2k-3}) := ( (k-1,2k-3),\ldots, (k-1,4k-7) ).
\end{align*}
Similarly $G''$ is is $(k-1)$-bottlenecked: we may take
\[N'' := ( (1,2k-3), (2,2k-3), \ldots, (k-1,2k-3) )\]
as our bottleneck. The network $G$ has underlying graph $G' \sqcup G'' / \sim$, the disjoint union of the graphs $G'$ and $G''$ where we mod out by the equivalence relation $\sim$, where $t'_{i+1} \sim s''_i$ for $ i \in [2k-3]$ and all other vertices are inequivalent. The sources and sinks of~$G$ are given by $S := (s'_1,\ldots,s'_{2k-1})$ and~$T := (t'_1,t''_1,\ldots,t''_{2k-3},t'_{2k-1}).$ We can witness that $G$ is $k$-bottlenecked by taking $N = N'$ and can witness that $G$ is $(k-1)$-sink-bottlenecked by taking~$N_T = N''$. The edge-weight function $\omega$ of $G$ is defined to be $1$ for all edges. From our construction we get $\mathrm{NCPath}_G(I,J) = \emptyset$ for~$I,J \in \binom{[2k-1]}{m}$ if and only if either $m \geq k+1$, or~$m = k$ and $\{1,2k-1\} \cap J = \emptyset$. This translates exactly to~$\phi(P_G) \in \mathcal{S}^{\geq 0}_{\mathcal{M}^*}$, as desired. \end{proof}

\section{Birational RSK and the octahedron recurrence} \label{sec:octa}

We now return to our original motivation. Let $X = (x_{ij})$ be an $m \times n$ matrix with entries in $\mathbb{R}_{>0}$. This matrix will be our input to birational RSK. Our output will be a three-dimensional array that, subject to the proper normalization, will obey the octahedron recurrence. Recall that we are interested in the weights of tuples of noncrossing paths in $X$. Therefore, in this section we will work with the rectangular grid $\Gamma_{m,n}$. We set the edge-weight function $\omega\colon E \to \mathbb{R}_{>0}$ of $\Gamma_{m,n}$ to be $\omega((i,j),(i',j')) := \sqrt{x_{ij}x_{i'j'}}$. For a path $\pi$ in $\Gamma_{m,n}$, we define a modified weight by $\widehat{\mathrm{wt}}(\pi) := \sqrt{x_{ij}x_{i'j'}}\cdot\mathrm{wt}(\pi)$ where $(i,j)$ is the start point of $\pi$ and $(i',j')$ is its endpoint. And for a tuple~$\Pi = (\pi_1,\ldots,\pi_k)$ of paths we define~$\widehat{\mathrm{wt}}(\Pi) := \prod_{i=1}^{k} \widehat{\mathrm{wt}}(\pi_i)$. Similarly for a pair of tuples of paths $(\Pi,\Sigma)$ we define~$\widehat{\mathrm{wt}}(\Pi,\Sigma) := \widehat{\mathrm{wt}}(\Pi)\cdot\widehat{\mathrm{wt}}(\Sigma)$. This modified weight is defined in this way so that~$\widehat{\mathrm{wt}}(\pi) = \prod_{(i,j) \in \mathrm{Vert}(\pi)} x_{ij}$.

Now we define a three dimensional array $\overline{Y} = (\overline{y}_{i,j,k})$ whose indices run over all $i,j,k \in \mathbb{Z}$ which satisfy $0 \leq k \leq \mathrm{min}(i,j)$ and for which
there exist some~$a,b \in \{0,1\}$ such that $(i+a,j+b) \in \Gamma_{m,n}$, where the entries of $\overline{Y}$ are given by~$\overline{y}_{i,j,k} := \sum_{\Pi \in \mathrm{RSKPath}(i,j,k)} \widehat{\mathrm{wt}}(\Pi)$ with
\[\mathrm{RSKPath}(i,j,k) := \mathrm{NCPath}_{\Gamma_{m,n}}(\{(1,1),...,(1,k)\}, \{(i,j-k+1),...,(i,j)\}).\]
If this sum is empty (which happens when any of $i$, $j$, or $k$ is zero) we treat it as~$1$. From the array $\overline{Y}$ we define a normalized array $\widetilde{Y} = (\widetilde{y}_{i,j,k})$ whose indices run over the same set as the indices of $Y$. For $(i,j) \in \Gamma_{m,n}$, let us define the \emph{rectangular product at $(i,j)$} to be~$\mathrm{rect}(i,j) :=  \prod_{r\le i, s\le j} x_{rs}$. If~$(i,j) \notin \Gamma_{m,n}$ we set~$\mathrm{rect}(i,j) := 1$. Then the entries of $\widetilde{Y}$ are given by $\widetilde{y}_{i,j,k} := \overline{y}_{i,j,k}/\mathrm{rect}(i,j)$. Theorem~\ref{thm:octa} will show that $\widetilde{Y}$ satisfies the octahedron recurrence, but first let us spell out the exact connection with birational RSK.

Greene's theorem says that if $\mathrm{mis}_{\sigma}(k)$ is the maximal size of a union of~$k$ disjoint increasing subsequences in $\sigma \in S_n$, then~$\mathrm{mis}_{\sigma}(k) = \lambda_1 + \ldots + \lambda_k$ where~$\lambda$ is the shape of the output of RSK applied to $\sigma$. Thus in order to obtain the~$\lambda_k$ from sizes of increasing subsequences in $\sigma$, we see~$\lambda_k = \mathrm{mis}_{\sigma}(k) - \mathrm{mis}_{\sigma}(k-1)$. In the birational setting, this means that to move from sums of weights of tuples of noncrossing paths back to RSK, we should take quotients of successive entries. So we define another three-dimensional array \mbox{$Y_{i,j,k} = (y_{i,j,k})$} whose indices run over all $i,j,k \in \mathbb{Z}$ which satisfy $0 \leq k \leq \mathrm{min}(i,j)+1$ and for which there exist~$a,b \in \{0,1,2\}$ such that $(i+a,j+b) \in \Gamma_{m,n}$, with entries $y_{i,j,0} := 1$, $y_{i,j,\mathrm{min}(i,j)+1} := 1$, and~$y_{i,j,k} := \overline{y}_{i,j,k} / {\overline{y}_{i,j,k-1}}$ for $0 < k < \mathrm{min}(i,j)+1$. The map~$X \mapsto Y$ could be called birational RSK. Alternatively, we might want RSK to map $X$ to another $m \times n$ matrix. In that case, define~$Z = (z_{ij})$ to be the~\mbox{$m \times n$} matrix with entries in $\mathbb{R}_{>0}$ as follows: for $(i,j) \in \Gamma_{m,n}$, let us set~$l(i,j) := \mathrm{min}(m-i,n-j)$; then $z_{ij} := y_{i+l(i,j),j+l(i,j),l(i,j)+1}$. In other words, $Z$ is obtained by flattening the outer border of $Y$. Then the map $X \mapsto Z$ is birational RSK as a map between matrices. Note that the array $Y$ can also be computed as follows: the boundary conditions are given by~$y_{i,j,0} = 1$ and~$y_{i,j,\mathrm{min}(i,j)+1} = 1$, and for $1 \leq k \leq \mathrm{min}(i,j)$ we have the recursive formula
\[y_{ijk} = \frac{x_{ijk}(y_{i-1,j,k} + y_{i,j-1,k})}{y_{i-1,j-1,k-1}(\frac{1}{y_{i-1,j,k-1}} + \frac{1}{y_{i,j-1,k-1}}) } \]
where $x_{ijk} := x_{ij}$ if $k = 1$ and $x_{ijk} := 1$ otherwise. We will not prove this recursive formula as it is tangential to our aims, but at any rate it follows immediately from Theorem~\ref{thm:octa} below. For those used to thinking about classical RSK in terms of insertion and bumping it may be rather unclear where this formula comes from. See~\cite{hopkins} for an expository development of the tropicalized version of this formula for classical RSK motivated by certain important properties of RSK such as symmetry with respect to transposition.

\begin{example}
Suppose $X = (x_{ij})_{1\leq i,j \leq 2}$.
Observe that we use standard matrix notation with $x_{11}$ in the upper-left corner. Then the three-dimensional arrays associated to $X$ are
\begin{center}
$\overline{Y} = $\setlength{\tabcolsep}{9pt}
\begin{tabular}{|c | c | c |}
\hline
$\begin{matrix} 1 & 1 & 1 \\ 1 & 1 & 1 \\ 1 & 1 & 1 \end{matrix}$ & $\begin{matrix} \\ x_{11} & x_{11}x_{12} \\  x_{11}x_{21} & x_{11}x_{22}(x_{12}+x_{21}) \end{matrix}$ & $\begin{matrix}  \\ \\ x_{11}x_{12}x_{21}x_{22} \end{matrix}$ \\
$k=0$ & $k=1$ & $k=2$  \\ \hline
\end{tabular} \\ \medskip

$\widetilde{Y}= $\setlength{\tabcolsep}{9pt}
\begin{tabular}{|c | c | c | }
\hline
$\begin{matrix} 1 & 1 & 1 \\ 1 & \frac{1}{x_{11}} & \frac{1}{x_{11}x_{12}} \\ 1 & \frac{1}{x_{11}x_{21}} & \frac{1}{x_{11}x_{12}x_{21}x_{22}} \end{matrix}$ & $\begin{matrix} \\ 1 & 1 \\  1 & \frac{1}{x_{12}} + \frac{1}{x_{21}} \end{matrix}$ & $\begin{matrix}  \\ \\ 1 \end{matrix}$  \\
$k=0$ & $k=1$ & $k=2$ \\ \hline
\end{tabular} \medskip

$Y =$ \setlength{\tabcolsep}{5pt}
\begin{tabular}{|c |c | c | c | }
\hline
$\begin{matrix} 1 & 1 & 1 & 1 \\ 1 & 1 & 1 & 1 \\ 1 & 1 & 1 & 1 \\ 1 & 1 & 1 & 1 \end{matrix}$ & $\begin{matrix} \\ 1 & 1 & 1 \\ 1 & x_{11} & x_{11}x_{12} \\ 1 & x_{11}x_{21} & x_{11}x_{22}(x_{12}+x_{21}) \end{matrix}$ & $\begin{matrix} \\ \\ 1 & 1 \\ 1 & \frac{x_{12}x_{21}}{x_{12}+x_{21} } \end{matrix}$ & $\begin{matrix} \\ \\ \\ 1 \end{matrix}$ \\
$k=0$ & $k=1$ & $k=2$ & $k=3$ \\ \hline
\end{tabular}\\
\end{center}
Here the bottom-right corner of each level in each array is aligned. Also,
\[ Z = \begin{pmatrix}\frac{x_{12}x_{21}}{x_{12}+x_{21} } & x_{11}x_{12} \\ x_{11}x_{21} & x_{11}x_{22}(x_{12}+x_{21})\end{pmatrix}. \]
\end{example}

\begin{thm} \label{thm:octa}
The three-dimensional array $\widetilde{Y} = (\widetilde{y}_{i,j,k})$ can be computed as follows: the boundary conditions are~$\widetilde{y}_{i,j,0} = 1/\mathrm{rect}(i,j) \textrm{ and } \widetilde{y}_{i,j,\mathrm{min}(i,j)} = 1$, and for $1 \leq k \leq \mathrm{min}(i,j) - 1$ we have the recursive formula
\[ \widetilde{y}_{i,j,k} \widetilde{y}_{i-1,j-1,k-1} = \widetilde{y}_{i-1,j,k}\widetilde{y}_{i,j-1,k-1} + \widetilde{y}_{i-1,j,k-1}\widetilde{y}_{i,j-1,k} \]
In other words, $\widetilde{Y}$ satisfies the (bounded) octahedron recurrence.
\end{thm}

\begin{proof} Of course~$\overline{y}_{i,j,0} = 1$ and~$\overline{y}_{i,j,\mathrm{min}(i,j)} = \mathrm{rect}(i,j)$ are equivalent boundary conditions.  We have~$\overline{y}_{i,j,0} = 1$ by definition. We have~$\overline{y}_{i,j,\mathrm{min}(i,j)} = \mathrm{rect}(i,j)$ because there is a single tuple of paths in~$\mathrm{RSKPath}(i,j,,\mathrm{min}(i,j))$ and it covers exactly those vertices in $\Gamma_{m,n}$ that are less than or equal to $(i,j)$.

Now let $1 \leq k \leq \mathrm{min}(i,j) - 1$. The key to proving the recursive condition is to show that
\[ \overline{y}_{ijk}\overline{y}_{i-1,j-1,k-1} = (\overline{y}_{i-1,j,k}\overline{y}_{i,j-1,k-1} + \overline{y}_{i-1,j,k-1}\overline{y}_{i,j-1,k})x_{ij}. \tag{*} \label{eqn:overu} \]
For $k=1$, we have $\overline{y}_{i-1,j-1,k-1} = \overline{y}_{i,j-1,k-1} =  \overline{y}_{i-1,j,k-1} = 1$ and~\eqref{eqn:overu} follows from the fact that every path connecting $(1,1)$ to $(i,j)$ goes through exactly one of~$(i-1,j)$ or $(i,j-1)$, and conversely any path to either $(i-1,j)$ or~$(i,j-1)$ can be uniquely extended to a path to $(i,j)$. Assume $k \geq 2$. For $(i,j) \in \Gamma_{m,n}$, define the \emph{increasing} and \emph{decreasing triangular products} of length $l$ at~$(i,j)$ as
\[
\mathrm{tri}^{+}(i,j,l) := \prod_{r = i}^{i+l-1} \; \prod_{s = j}^{j+i+l-r-1}x_{rs} \text{\ and\ }
\mathrm{tri}^{-}(i,j,l) :=  \prod_{r = i-l+1}^{i} \; \prod_{s = j+i-l-r+1}^{j}x_{rs}.
\]
The first equation makes sense for $1\leq l \leq \min(m-i+1,n-j+1)$, and the second equation makes sense for $1\leq l\leq \min(i,j)$. Consider the network $\Gamma_{i,j}^{k}$ with edge-weight function $\omega$ the same as for $\Gamma_{m,n}$ above. Set $I,J := \{2,4,\ldots,2k-2\}$ and~$\kappa := \mathrm{tri}^{+}(1,1,k-2) \cdot  \mathrm{tri}^{+}(1,1,k-1)$. Then there is a bijection
\[ \varphi\colon \mathrm{PNCPath}_{\Gamma_{i,j}^{k}}(I,J) \to \mathrm{RSKPath}(i-1,j-1,k-1) \times \mathrm{RSKPath}(i,j,k)\]
such that
\[ \widehat{\mathrm{wt}}(R,B) \cdot \kappa \cdot  \mathrm{tri}^{-}(i-1,j-1,k-2) \cdot \mathrm{tri}^{-}(i,j,k-1) = \widehat{\mathrm{wt}}(\varphi(R,B)).\]
Specifically, if $(R,B) = ( (r_1,\ldots,r_{k-1}), (b_1,\ldots,b_k) ) \in  \mathrm{PNCPath}_{\Gamma_{i,j}^{k}}(I,J)$ then we define $\varphi(R,B) := ( (\pi_1,\ldots,\pi_{k-1}), (\sigma_1,\ldots,\sigma_{k}))$ where
\begin{align*}
\pi_s &:= \{(t,s)\}_{t=1}^{k-1-s} \cdot r_s \cdot \{(i-s+t,j-k+s)\}_{t=1}^{s-1} \\
\sigma_s &:=  \{(t,s)\}_{t=1}^{k-s} \cdot b_s \cdot \{(i-s+t,j-k+s)\}_{t=2}^{s}.
\end{align*}
(Here $\cdot$ denotes concatenation of sequences.) In other words, $\varphi$ extends the paths vertically to connect to the appropriate start and end points for paths in~$\mathrm{RSKPath}(i-1,j-1,k-1)$ and $ \mathrm{RSKPath}(i,j,k)$; there is a unique way to do this. Similarly, if we set $J' := \{1,3,\ldots,2k-3\}$ then there is a bijection
\[ \varphi'\colon \mathrm{PNCPath}_{\Gamma_{i,j}^{k}}(I,J') \to \mathrm{RSKPath}(i,j-1,k-1) \times \mathrm{RSKPath}(i-1,j,k)\]
such that
\[ \frac{\widehat{\mathrm{wt}}(R',B')}{x_{i,j-k+1}} \cdot \kappa \cdot \mathrm{tri}^{-}(i,j-1,k-2) \cdot \mathrm{tri}^{-}(i-1,j,k-1) = \widehat{\mathrm{wt}}(\varphi'(R',B')).\]
Here for $(R',B') = ( (r'_1,\ldots,r'_{k-1}), (b'_1,\ldots,b'_k) ) \in  \mathrm{PNCPath}_{\Gamma_{i,j}^{k}}(I,J')$ we define~$\varphi'(R',B') := ( (\pi'_1,\ldots,\pi'_{k-1}), (\sigma'_1,\ldots,\sigma'_{k}) )$ where
\begin{align*}
\pi'_s &:= \{(t,s)\}_{t=1}^{k-1-s} \cdot r'_s \cdot \{(i-s+t,j-k+s)\}_{t=2}^{s} \\
\sigma'_s &:=  \{(t,s)\}_{t=1}^{k-s} \cdot b'_s \cdot \{(i-s+t,j-k+s)\}_{t=1}^{s-1}.
\end{align*}
Again, $\varphi'$ just extends paths vertically. And if we set $J'' := \{3,5,\ldots,2k-1\}$ then there is a bijection
\[ \varphi''\colon \mathrm{PNCPath}_{\Gamma_{i,j}^{k}}(I,J'') \to \mathrm{RSKPath}(i-1,j,k-1) \times \mathrm{RSKPath}(i,j-1,k)\]such that
\[ \frac{\widehat{\mathrm{wt}}(R'',B'')}{x_{i-k+1,j}} \cdot \kappa \cdot  \mathrm{tri}^{-}(i-1,j,k-2) \cdot \mathrm{tri}^{-}(i,j-1,k-1) = \widehat{\mathrm{wt}}(\varphi''(R'',B'')).\]
Here for $(R'',B'') = ( (r''_1,\ldots,r''_{k-1}), (b''_1,\ldots,b''_k) ) \in  \mathrm{PNCPath}_{\Gamma_{i,j}^{k}}(I,J'')$ such that $b''_1 = \{v_t\}_{t=0}^{l}$ we define~$\varphi''(R'',B'') := ( (\pi''_1,\ldots,\pi''_{k-1}), (\sigma''_1,\ldots,\sigma''_{k}) )$ where
\begin{align*}
\pi_s &:= \{(t,s)\}_{t=1}^{k-1-s} \cdot r_s \cdot \{(i-s+t,j-k+s+1)\}_{t=1}^{s-1} \\
\sigma_s &:=  \begin{cases} \{(t,1)\}_{t=1}^{k-1} \cdot \{v_t\}_{t=0}^{l-1} & \textrm{if $s=1$} \\ \{(t,s)\}_{t=1}^{k-s} \cdot b_s \cdot \{(i-s+t,j-k+s-1)\}_{t=2}^{s} & \textrm{otherwise}. \end{cases}
\end{align*}
Now $\varphi''$ has to slide the end point of $b_1$ to the left, but the other paths it again just extends vertically. Corollary~\ref{cor:coroftswap2} tells us that
\[ \sum_{(R,B) \in  \mathrm{PNCPath}(I,J)} \hspace{0pt minus 1fil} \widehat{\mathrm{wt}}(R,B) = \hspace{0pt minus 1fil} \sum_{(R',B') \in  \mathrm{PNCPath}(I,J')} \hspace{0pt minus 1fil} \widehat{\mathrm{wt}}(R',B')  +  \hspace{0pt minus 1fil} \sum_{(R'',B'') \in \mathrm{PNCPath}(I,J'')} \hspace{0pt minus 1fil} \widehat{\mathrm{wt}}(R'',B'')\]
and together with
\begin{align*}
x_{ij} &= \frac{ \mathrm{tri}^{-}(i-1,j-1,k-2) \cdot \mathrm{tri}^{-}(i,j,k-1) \cdot x_{i,j-k+1}}{  \mathrm{tri}^{-}(i-1,j,k-2) \cdot \mathrm{tri}^{-}(i,j-1,k-1)}  \\
&= \frac{ \mathrm{tri}^{-}(i-1,j-1,k-2) \cdot \mathrm{tri}^{-}(i,j,k-1) \cdot x_{i-k+1,j} }{  \mathrm{tri}^{-}(i,j-1,k-2) \cdot \mathrm{tri}^{-}(i-1,j,k-1)}
\end{align*}
we conclude that indeed equation~\eqref{eqn:overu} holds. To finish, we compute
\begin{align*}
\widetilde{y}_{ijk} = \frac{\overline{y}_{ijk}}{\mathrm{rect}(i,j)}
&= \frac{x_{ij}(\overline{y}_{i-1,j,k}\overline{y}_{i,j-1,k-1}  + \overline{y}_{i,j-1,k} \overline{y}_{i-1,j,k-1} )}{\mathrm{rect}(i,j)\cdot \overline{y}_{i-1,j-1,k-1}}\\
&= \frac{\mathrm{rect}(i-1,j-1) \cdot (\overline{y}_{i-1,j,k}\overline{y}_{i,j-1,k-1}  + \overline{y}_{i,j-1,k} \overline{y}_{i-1,j,k-1} )}{\mathrm{rect}(i-1,j)\cdot \mathrm{rect}(i,j-1)\cdot \overline{y}_{i-1,j-1,k-1}}\\
&= \frac{\widetilde{y}_{i-1,j,k}\widetilde{y}_{i,j-1,k-1}  + \widetilde{y}_{i,j-1,k} \widetilde{y}_{i-1,j,k-1}}{\widetilde{y}_{i-1,j-1,k-1}}.
\end{align*}
Thus, $\widetilde{Y}$ satisfies the octahedron recurrence~\cite{speyer}~\cite{henriques}. \end{proof}

\section{Schur functions and Schur positivity} \label{sec:schur}

In this section we apply our network path weight relations to the problem of finding identities for products of Schur functions. The identities we obtain are reminiscent of those obtained by Fulmek and Kleber~\cite{fulmek}, who also used path swapping. We then apply the identities to demonstrate Schur positivity for certain expressions involving products of Schur functions. Here we assume familiarity with partitions, Young tableaux, and the ring of symmetric functions. A reference is Stanley~\cite[\S7]{stanley2} and we will generally follow Stanley's notation. One notational remark is that we use~$c^r$ to denote the rectangular partition with~$r$ rows of length $c$, and we also sometimes write expressions like~$(c_1^{r_1},c_2^{r_2},c_3^{r_3})$ to denote the partition with~$r_1$ rows equal to $c_1$, $r_2$ rows equal to~$c_2$, and $r_3$ rows equal to $c_3$, where~$c_1 > c_2 > c_3$. For $T$ a semistandard Young tableau (SSYT) let us define the monomial $x^T := \prod_{i=1}^{\infty} x_i^{m(i,T)}$ where $m(i,T)$ is the number of entries equal to $i$ in $T$. Recall that the Schur function $s_\lambda$ can be defined combinatorially by~$s_\lambda(x) := \sum_{T} x^T$ where the sum is over all SSYTs~$T$ of shape~$\lambda$. We will need various specializations of Schur functions to state our results. We define $s_\lambda^{X}$, where $X\subseteq \mathbb{Z}_{>0}$, to be $s_\lambda(x_1,x_2,...)$ with specializations~$x_i=0$ for~$i\notin X$. Let us call an SSYT whose entries are among~$[n]$ an \emph{$n$-tableau}. Then it is clear from the combinatorial definition of Schur functions that $s^{[n]}_{\lambda} = \sum_T x^T$ where the sum is over all $n$-tableaux $T$ of shape $\lambda$.

We now recall an equivalent definition of Schur functions in terms of nonintersecting paths. Already Gessel and Viennot~\cite{gessel} were aware of the connection between tableaux and nonintersecting lattice paths in $\mathbb{Z}^2$.  Let us make~$\mathbb{Z}^2$ into a graph with horizontal edges $((i,j),(i-1,j))$ and vertical edges  $((i,j),(i,j-1))$. We now use Cartesian coordinates for $\mathbb{Z}^2$ so~$(-\infty,-\infty)$ will be in the bottom-left corner. Although $\mathbb{Z}^2$ is infinite, this is no problem for us as we will only ever use a finite portion of it. We set the edge-weight function $\omega$ of $\mathbb{Z}^2$ to be~$\omega((i,j), (i-1,j)) := x_j$ for horizontal edges and $\omega((i,j), (i,j-1)) := 1$ for vertical edges. Let~$\lambda = (\lambda_1,\ldots,\lambda_k)$ be a partition. For $n \geq 1$ let
\[\mathrm{SPath}(\lambda,n) := \mathrm{NCPath}_{\mathbb{Z}^2}(\{(\lambda_{k+1-i}+i,n)\}_{i=1}^{k},\{(i,1)\}_{i=1}^{k}).\]
Then $s^{[n]}_\lambda = \sum_{\Pi \in\mathrm{SPath}(\lambda,n) } \mathrm{wt}(\Pi)$, which follows from a simple bijection between $n$-tableaux of shape $\lambda$ and paths in $\mathrm{SPath}(\lambda,n)$ (see~\cite[Theorem 7.16.1]{stanley2}). In fact, we obtain the following by translation:

\begin{prop} \label{prop:schurpaths}
For $a,b,c \in \mathbb{Z}$ with $1 \leq a \leq b$ let
\begin{align*}
\mathrm{SPath}^c(\lambda,a,b) &:= \mathrm{NCPath}_{\mathbb{Z}^2}(\{(\lambda_{k+1-i}+i+c,b)\}_{i=1}^{k},\{(i+c,a)\}_{i=1}^{k}).
\end{align*}
Then $s^{[a,b]}_\lambda = \sum_{\Pi \in\mathrm{SPath}^c(\lambda,a,b) } \mathrm{wt}(\Pi)$.
\end{prop}

Our main result in this section is the following identity of Schur functions:
\begin{thm}\label{thm:schuridentity}
Let $\lambda=(\lambda_1,\lambda_2,\ldots,\lambda_k)$ and $0 \leq t \leq k-1$. Then
\begin{align*}
s_{\lambda} s^{[2,\infty)}_{(\lambda_1,\ldots,\lambda_t,\lambda_{t+2}-1,\ldots,\lambda_k-1)} = & s^{[2,\infty)}_{\lambda} s_{(\lambda_1,\ldots,\lambda_t,\lambda_{t+2}-1,\ldots,\lambda_k-1)} \\
&+x_1s_{(\lambda_1-1,\ldots,\lambda_k-1)}s^{[2,\infty)}_{(\lambda_1+1,\ldots,\lambda_t+1,\lambda_{t+2},\ldots,\lambda_k)}.
\end{align*}
\end{thm}

\begin{proof} In order to prove this identity we use an interlacing network~$G$. Fix some~\mbox{$n \geq k$}. For $i \in [k]$ define $v_i := (\lambda_{k+1-i}+i,n) \in \mathbb{Z}^2$. Define $G$ to be the network whose underlying graph is the subgraph of $\mathbb{Z}^2$ with vertices in the rectangle between $(1,1)$ and~$(\lambda_1 + k,n)$ and with sources
\[ S = (s_1,\ldots,s_{2k-1}) := ( v_1,v_1,v_2,v_2,\ldots,v_{k-t},\overline{v_{k-t}},\ldots,v_{k},v_{k})\]
(where the overline denotes omission) and sinks
\[ T = (t_1,\ldots,t_{2k-1}) := ( (1,k), (2,k), (2,k-1), (3,k-1),\ldots, (k,2),(k,1)). \]
To witness that $G$ is interlacing we may take $N = \{s_1,s_3,\ldots,s_{2k-1}\}$ as a~$k$-bottleneck and $N_T = \{t_2,t_4,\ldots,t_{2k-2}\}$ as a $(k-1)$-sink bottleneck. Figure~\ref{fig:schur2} illustrates $G$ together with an element of $\mathrm{PNCPath}(G)$ for some specific parameters $\lambda$, $t$ and $n$.

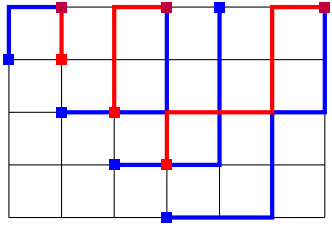
\begin{figure}
\begin{center}
\begin{tikzpicture}[scale=0.7]
\foreach \x in {1,...,5}{
	\draw[thin] (1,\x) -- (7,\x);
}
\foreach \y in {1,...,7}{
	\draw[thin] (\y,1) -- (\y,5);
}
\draw[color=blue,ultra thick] (1,4) -- (1,5) -- (2,5);
\node[rectangle,fill=blue,inner sep=2pt] at (1,4) {};
\draw[color=blue,ultra thick] (2,3) -- (4,3) -- (4,5);
\node[rectangle,fill=blue,inner sep=2pt] at (2,3) {};
\draw[color=blue,ultra thick] (3,2) -- (5,2) -- (5,5);
\node[rectangle,fill=blue,inner sep=2pt] at (3,2) {};
\node[rectangle,fill=blue,inner sep=2pt] at (5,5) {};
\draw[color=blue,ultra thick] (4,1) -- (6,1) -- (6,3) -- (7,3) -- (7,5);
\node[rectangle,fill=blue,inner sep=2pt] at (4,1) {};
\draw[color=red,ultra thick] (2,4) -- (2,5);
\node[rectangle,fill=red,inner sep=2pt] at (2,4) {};
\node[rectangle,fill=purple,inner sep=2pt] at (2,5) {};
\draw[color=red,ultra thick] (3,3) -- (3,5) -- (4,5);
\node[rectangle,fill=red,inner sep=2pt] at (3,3) {};
\node[rectangle,fill=purple,inner sep=2pt] at (4,5) {};
\draw[color=red,ultra thick] (4,2) -- (4,3) -- (6,3) -- (6,5) -- (7,5);
\node[rectangle,fill=red,inner sep=2pt] at (4,2) {};
\node[rectangle,fill=purple,inner sep=2pt] at (7,5) {};
\end{tikzpicture}
\end{center}
\caption{ For $\lambda = (3,2,2,1)$, $t =1$, and $n=5$: the network~$G$ and an element of~$\mathrm{PNCPath}_G(\{2,4,6\},\{2,4,6\})$.} \label{fig:schur2}
\end{figure}

To simplify notation, set
\begin{align*}
\mu &:= (\lambda_1,\ldots,\lambda_t,\lambda_{t+2}-1,\ldots,\lambda_k-1) \\
\nu &:= (\lambda_1+1,\ldots,\lambda_t+1,\lambda_{t+2},\ldots,\lambda_k) \\
\rho &:= (\lambda_1-1,\ldots,\lambda_k-1).
\end{align*}
Let $I,J := \{2,4,\ldots,2k-2\}$ and $J' := [1,2k-2] \setminus J$ and $J'' := [2,2k-1] \setminus J$. Then there are bijections
\begin{align*}
\varphi\colon \mathrm{PNCPath}_G(I,J) &\to  \mathrm{SPath}^1(\mu,2,n) \times \mathrm{SPath}^0(\lambda,1,n) \\
\varphi'\colon \mathrm{PNCPath}_G(I,J') &\to \mathrm{SPath}^0(\nu,2,n) \times \mathrm{SPath}^1(\rho,1,n) \\
\varphi''\colon \mathrm{PNCPath}_G(I,J'') &\to \mathrm{SPath}^0(\mu,1,n) \times \mathrm{SPath}^0(\lambda,2,n)
\end{align*}
such that
\begin{align*}
\mathrm{wt}(R,B) &= \mathrm{wt}(\varphi(R,B)) \\
\mathrm{wt}(R',B') &= x_1\cdot \mathrm{wt}(\varphi'(R',B'))\\
\mathrm{wt}(R'',B'') &= \mathrm{wt}(\varphi''(R'',B''))
\end{align*}
for all appropriate $(R,B), (R',B'), (R'',B'') \in\mathrm{PNCPath}(G)$. These bijections have a very similar description to those in the proof of Theorem~\ref{thm:octa}: the maps~$\varphi$ and $\varphi''$ merely extend the paths vertically to reach the necessary start and end points; $\varphi'$ also just extends paths vertically, except for $b'_{2k-1}$ (the rightmost blue path) which it moves to the right, thus accounting for the factor of $x_1$. Corollary~\ref{cor:coroftswap2} tells us that
\[ \sum_{(R,B) \in  \mathrm{PNCPath}(I,J)} \hspace{0pt minus 1fil} \widehat{\mathrm{wt}}(R,B) = \hspace{0pt minus 1fil} \sum_{(R',B') \in  \mathrm{PNCPath}(I,J')} \hspace{0pt minus 1fil} \widehat{\mathrm{wt}}(R',B')  +  \hspace{0pt minus 1fil} \sum_{(R'',B'') \in \mathrm{PNCPath}(I,J'')} \hspace{0pt minus 1fil} \widehat{\mathrm{wt}}(R'',B'')\]
and together with Proposition~\ref{prop:schurpaths} we conclude $s^{[2,n]}_{\mu}s^{[n]}_{\lambda} = x_1s^{[2,n]}_{\nu}s^{[n]}_{\rho} + s^{[n]}_{\mu}s^{[2,n]}_{\lambda}$. Taking the limit~$n \to \infty$ gives us the result. \end{proof}

By taking $t=k-1$ in Theorem~\ref{thm:schuridentity}, we get the following corollaries.

\begin{cor}\label{cor:schuridentityspecialcase}
For $\lambda = (\lambda_1,\ldots,\lambda_k)$,
\[s_{\lambda} s^{[2,\infty)}_{(\lambda_1,\ldots,\lambda_{k-1})} = s_{\lambda}^{[2,\infty)}s_{(\lambda_1,\ldots,\lambda_{k-1})}+ x_1s_{(\lambda_1-1,\ldots,\lambda_k-1)} s_{(\lambda_1+1,\ldots,\lambda_{k-1}+1)}^{[2,\infty)}.\]
\end{cor}

\begin{cor}\label{cor:rectengelschuridentity}
For any $c, r \geq 1$, $s_{c^r}s_{c^{r-1}}^{[2,\infty)} = s_{c^{r-1}}s_{c^r}^{[2,\infty)} + x_1s_{(c-1)^r}s_{(c+1)^{r-1}}^{[2,\infty)}$.
\end{cor}

Compare Corollary~\ref{cor:rectengelschuridentity} to the following result of Kirillov~\cite{kirillov1}:

\begin{thm}[Kirillov] \label{thm:kirillov}
For any $c, r \geq 1$, $(s_{c^r})^2 = s_{c^{r-1}} s_{c^{r+1}} + s_{(c-1)^r}s_{(c+1)^r}$.
\end{thm}

Fulmek and Kleber~\cite{fulmek} give a bijective proof of this identity; indeed, they prove a more general identity, which we state below. Their proof also goes through a certain algorithm that swaps pairs of tuples of nonintersecting paths. In fact, their notion of \emph{changing tail} is quite similar to the path visiting the vertices $v_0, w_1, v_1, \ldots, v_c, w_c$ we build as part of the algorithm defining $\tau$ in~\S\ref{sec:sinkswap}. However, there are significant differences: for one, their networks are not interlacing (and so they never use bottlenecks); also, their procedure changes the size of each tuple, whereas ours does not. The result is that our identities oddly involve Schur functions in different sets of variables.

We now explain how these three-term Schur function identities, those due to Kirillov, Fulmek-Kleber, and our own, lead to some results about Schur positivity. Recall that we say that a symmetric function is \emph{Schur positive} if it has all nonnegative coefficients in the basis of Schur functions. For two symmetric functions $f$ and $g$, we write $f \geq_s g$ if the difference $f-g$ is Schur positive. There has been some interest in understanding when we have $s_{\nu}s_{\rho} \geq_s s_{\lambda}s_{\mu}$ for partitions $\nu, \rho, \lambda, \mu$. If we let $c_{\lambda,\mu}^{\alpha}$ be the Littlewood--Richardson coefficients given by $s_{\lambda}s_{\mu} = \sum_{\alpha} c_{\lambda,\mu}^{\alpha} s_{\alpha}$, this question is equivalent to the question of when we have~$c_{\nu,\rho}^{\alpha} \geq c_{\lambda,\mu}^{\alpha}$ for all $\alpha$. Research on this problem has focused on the case where the partitions $\nu$ and $\rho$ are thought of as ``functions'' of $\lambda$ and~$\mu$ as in~\cite{fomin}~\cite{lam}~\cite{bergeron}~\cite{stembridge} (see also related work~\cite{shimozono1}~\cite{shimozono2}~\cite{schilling} for $q$-analogs of this problem). We will now state a Schur positivity conjecture of this form. This conjecture was communicated to us privately by Alex Postnikov, who discovered it in collaboration with Pavlo Pylyavskyy and Thomas Lam (see also the papers~\cite{dobrovolska}~\cite{chari} which investigate this conjecture).

\begin{conj}[Lam-Postnikov-Pylyavskyy]\label{conj:schurpos}
Let $\nu = (\nu_1,\ldots,\nu_n)$ be a partition. For all $\sigma \in \mathfrak{S}_n$, $i \in [n]$, and choices $\pm \in \{+,-\}$, define a new sequence~$\nu^{\pm}(\sigma,i) = (\nu^{\pm}(\sigma,i)_1,\ldots,\nu^{\pm}(\sigma,i)_n)$ by
\[\nu^{\pm}(\sigma,i)_j := \begin{cases} \nu_j \pm 1 & \textrm{ if $\sigma^{-1}(j) \in [i]$} \\
\nu_j & \textrm{ otherwise}.\end{cases} \]

Let $\lambda = (\lambda_1,\ldots,\lambda_n)$ and $\mu = (\mu_1,\ldots,\mu_n)$ be two partitions and let their difference vector be~$\delta = (\delta_1,\ldots,\delta_n) := (\lambda_1 - \mu_1,\ldots,\lambda_n - \mu_n)$. Let $\sigma \in \mathfrak{S}_n$ be the unique permutation so that $\delta_{\sigma(1)} \geq \cdots \geq \delta_{\sigma(n)}$ and $\delta_{\sigma(i)} = \delta_{\sigma(j)}$ for~$i < j$ implies that~$\sigma(i) < \sigma(j)$. Set~$\mathcal{D} := \{i \in [n]\colon \delta_{\sigma(i)}>0 \textrm{ and } (i=n \textrm{ or } \delta_{\sigma(i)} > \delta_{\sigma(i+1)})\}$. Then for all $i \in \mathcal{D}$ we have $s_{\lambda^{-}(\sigma,i)}s_{\mu^{+}(\sigma,i)} \geq_s s_{\lambda}s_{\mu}$.
\end{conj}

That $\lambda^{-}(\sigma,i)$ and $\mu^{+}(\sigma,i)$ remain partitions for all $i \in \mathcal{D}$ in Conjecture~\ref{conj:schurpos} just requires checking some cases. The three-term Schur function identities we have been studying in this section resolve some special cases of this conjecture. For instance, by applying $\frac{\partial}{\partial x_1}$ to both sides and setting ${x_1=0}$ in Theorem~\ref{thm:schuridentity}, we get the following identity of Schur functions that all use the same set of variables, albeit involving skew Schur functions.

\begin{cor}\label{cor:schurderivate}
For $\lambda = (\lambda_1,\ldots,\lambda_k)$ and $0 \leq t \leq k-1$,
\begin{align*}
s_{\lambda/ 1}s_{(\lambda_1,\ldots,\lambda_t,\lambda_{t+2}-1,\ldots,\lambda_k-1 )} = &s_{(\lambda_1,\ldots,\lambda_t,\lambda_{t+2}-1,\ldots,\lambda_k-1)/ 1}s_\lambda \\
&+ s_{(\lambda_1-1,\ldots,\lambda_k-1)}s_{(\lambda_1+1,\ldots,\lambda_t+1,\ldots,\lambda_k)}.
\end{align*}
Here $\nu/1$ denotes the skew shape of $\nu$ minus its top-leftmost box.
\end{cor}

But we can in fact obtain the following Schur positivity result that concerns only regular Schur functions.

\begin{prop}\label{prop:postnikovspecialcase}
Let $c,r\geq1$ and $0 \leq t \leq r-1$. Then
\[s_{(c^{r-1},c-1)}s_{(c^t,(c-1)^{r-t-1} )}-s_{(c-1)^{r}}s_{((c+1)^t,c^{r-t-1})}\]
is Schur positive.
\end{prop}

This proposition is a special case of Conjecture~\ref{conj:schurpos}. In order to see why, let~$\lambda=((c+1)^t,c^{r-t-1})$ and $ \mu=(c-1)^{r}$. Then $\delta=(2^t, 1^{r-t-1},-(c-1))$ and so $\sigma$ is the identity permutation. Note $r-1 \in \mathcal{D}$, so with~$i=r-1$ we get $\lambda^{-}(\sigma,r-1)= (c^t,(c-1)^{r-t-1})$ and $\mu^{+}(\sigma,r-1)=(c^{r-1},c-1)$. The conjecture says we should have $s_{\lambda^{-}(\sigma,r-1)}s_{\mu^{+}(\sigma,r-1)} \geq_s s_{\lambda}s_{\mu}$, which is exactly what Proposition~\ref{prop:postnikovspecialcase} asserts.

\begin{proof}[Proof of Proposition~\ref{prop:postnikovspecialcase}]: Applying Corollary~\ref{cor:schurderivate} to the case in which $\lambda$ is the rectangular partition $c^{r}$, and using the skew version of Pieri's rule \cite[Corollary 7.5.19]{stanley2} leads us to the following three cases:
\begin{enumerate}
\item If $1 \leq t \leq r-2$ then $s_{(c^{r-1},c-1)}s_{(c^t,(c-1)^{r-t-1})}$ is equal to
\[\Big [s_{(c^{t-1},(c-1)^{r-t})}+s_{(c^t,(c-1)^{r-t-2},c-2)} \Big] s_{c^{r}}+s_{(c-1)^{r}}s_{((c+1)^t,c^{r-t-1})}.\]
\item If $t=0$ then $s_{(c^{r-1},c-1)}s_{(c-1)^{r-1}} = s_{((c-1)^{r-2},c-2)}s_{c^{r}} + s_{(c-1)^{r}}s_{c^{r-1}}$.
\item If $t=r-1$ then $s_{(c^{r-1},c-1)}s_{c^{r-1}} = s_{(c^{r-2},c-1)}s_{c^{r}} + s_{(c-1)^{r}}s_{(c+1)^{r-1}}$.
\end{enumerate}
Thus, because products of Schur functions are Schur positive (in other words, because Littlewood--Richardson coefficients are nonnegative) we are done. \end{proof}

Another example of a special case of Conjecture~\ref{conj:schurpos} is obtained from the following identity of Fulmek and Kleber \cite{fulmek}, which we mentioned earlier implies Theorem~\ref{thm:kirillov}.

\begin{thm}[Fulmek and Kleber]\label{thm:fulmekkleberidentity}
Let $\nu = (\nu_1, \nu_2,\ldots, \nu_{k+1})$ be a partition with~$k \geq 1$. Then
\[s_{(\nu_1, \ldots, \nu_{k})}s_{(\nu_2,\ldots, \nu_{k+1})} = s_{(\nu_2,\ldots, \nu_{k})}s_{(\nu_1, \ldots \nu_{k+1})} + s_{(\nu_2-1,\ldots, \nu_{k+1}-1)} s_{(\nu_1+1,\ldots, \nu_{k}+1)}.\]
\end{thm}

Setting $\lambda=(\nu_1+1,\ldots, \nu_{k}+1)$ and~$\mu=(\nu_2-1,\ldots, \nu_{k+1}-1)$ in Conjecture~\ref{conj:schurpos}, we get that all the elements of $\delta$ are positive. Let $\sigma$ be as in the conjecture. Then we see~$\lambda^{-}(\sigma,k) = (\nu_1,\ldots, \nu_{k})$ and $\mu^{+}(\sigma,k) =(\nu_2,\ldots, \nu_{r+1})$. Since $k \in \mathcal{D}$, the conjecture says we should have~$s_{\lambda^{-}(\sigma,k) }s_{\mu^{+}(\sigma,k)} \geq_s s_{\lambda}s_{\mu}$, which indeed follows from Theorem~\ref{thm:fulmekkleberidentity}.

Because the involution $\tau$ of \S3 makes sense not just for interlacing networks, but also more generally for $k$-bottlenecked networks, it can actually be applied in a different way to obtain another result about Schur positivity. In fact, $\tau$ leads to the proof of a different special case of Conjecture~\ref{conj:schurpos}. First we prove another (multi-term) Schur function identity. The following identity appeared earlier in \cite{gurevich} (Proposition 3.1 and Corollary 3.2). It is also a consequence of Lemma 16 in \cite{fulmek}. Our proof is independent of the above and uses the properties of our involution $\tau$.

\begin{thm} \label{thm:schurid2}
Let $\lambda=(\lambda_1,...,\lambda_k)$ and $\mu=(\mu_1,...,\mu_{k-1})$ be partitions that interlace in the sense that $\lambda_i\geq\mu_i\geq\lambda_{i+1}$ for $i \in [k-1]$. For $1\leq i\leq k$, define~$\lambda^i=(\lambda^i_1,...,\lambda^i_k)$ and $\mu^i=(\mu^i_1,...,\mu^i_{k-1})$ to be
\[ \lambda^i_j :=
\begin{cases}
\mu_j -1 &\text{ if } j<i \\
\lambda_i &\text{ if } j=i \\
\mu_{j-1}  &\text{ if } j>i
\end{cases}
\ \text{ and }\
\mu^i_j :=
\begin{cases}
\lambda_j  +1	&\text{ if } j<i \\
\lambda_{j+1}  &\text{ if } j\geq i.
\end{cases}
\]
Then we have
$ s_\lambda s_\mu = \sum_{i=1}^k s_{\lambda^i} s_{\mu^i}$
where $s_\nu$ is taken to be $0$ if $\nu$ is not a partition.
\end{thm}
\begin{proof}  In order to prove this identity we use a $k$-bottlenecked network~$G$. Fix~$n \geq k$. For $i \in [k]$ define~$v_i := (\lambda_{k+1-i}+i,n) \in \mathbb{Z}^2$ and for~$i \in [k-1]$ define~$u_i := (\mu_{k-i}+i,n)$. Define $G$ to be the network whose underlying graph is the subgraph of $\mathbb{Z}^2$ with vertices in the rectangle between~$(1,1)$ and~$(\lambda_1 + k,n)$ and with sources and sinks
\small
\begin{align*}
S &= (s_1,\ldots,s_{2k-1}) := ( v_1,u_1,v_2,u_2,\ldots,v_{k-1},u_{k-1},v_{k})\\
T &= (t_1,\ldots,t_{2k-1}) := ( (1,1),(1,1),(2,1),(2,1),\ldots,(k-1,1),(k-1,1),(k-1)).
\end{align*}
\normalsize
To witness that $G$ is $k$-bottlenecked we may take $N = \{t_1,t_3,\ldots,t_{2k-1}\}$.

Let $I,J := \{2,4,\ldots,2k-2\}$ and define $I^i := [2k-1] \setminus (\{2i-1\} \cup I)$ for all~$i \in [k]$. Then we have
\begin{align*}
\mathrm{PNCPath}_G(I,J) &= \mathrm{SPath}(\mu,n) \times \mathrm{SPath}(\lambda,n) \\
\mathrm{PNCPath}_G(I^i,J) &= \mathrm{SPath}(\mu^i,n) \times \mathrm{SPath}(\lambda^i,n)
\end{align*}
for all $i \in [k]$. Also, Remark~\ref{rem:sourceswap} tells us that $\mathrm{wt}(I,J) = \sum_{i=1}^{k} \mathrm{wt}(I^i,J)$. So we conclude $s^{[n]}_\lambda s^{[n]}_\mu = \sum_{i=1}^k s^{[n]}_{\lambda^i} s^{[n]}_{\mu^i}$. Taking~$n \to \infty$ gives us the result. \end{proof}

Note that Theorem~\ref{thm:schurid2} implies Theorem~\ref{thm:fulmekkleberidentity} by taking $\lambda = (\nu_1,\ldots,\nu_k,0)$ and $\mu = (\nu_2,\ldots,\nu_{k+1})$ (most terms are $0$). Similarly, taking~$\lambda = (\nu_1,\ldots,\nu_k)$ and $\mu = (\nu_1,\ldots,\nu_{t-1},\nu_{t+1},\ldots,\nu_k)$ in Theorem~\ref{thm:schurid2} for some $1 \leq t \leq k$ yields the following.

\begin{cor} \label{cor:schurpos2}
Let $\nu = (\nu_1,\ldots,\nu_k)$ be a partition and let $1 \leq t \leq k$. Then
\[ s_{\nu}s_{(\nu_1,\ldots,\nu_{t-1},\nu_{t+1},\ldots,\nu_k)} - s_{(\nu_1-1,\ldots,\nu_{t-1}-1,\nu_t,\ldots,\nu_k)}s_{(\nu_1+1,\ldots,\nu_{t-1}+1,\nu_{t+1},\ldots,\nu_{k})}\]
is Schur positive.
\end{cor}

To see why Corollary~\ref{cor:schurpos2} is a special case of Conjecture~\ref{conj:schurpos}, we can take $\lambda = (\nu_1+1,\ldots,\nu_{t-1}+1,\nu_{t+1},\ldots,\nu_{k})$ and $\mu = (\nu_1-1,\ldots,\nu_{t-1}-1,\nu_t,\ldots,\nu_k)$. Let $\sigma$ be as in that conjecture. Note that $\delta_i = 2$ for $i \leq t-1$ and~$\delta_i < 0$ for~$i \geq t$, so $\sigma(i) = i$ for $i \leq t-1$ and $t-1 \in \mathcal{D}$. Then the conjecture predicts~$s_{\lambda^{-}(\sigma,t-1)}s_{\mu^{+}(\sigma,t-l)} \geq_s s_{\lambda}s_{\mu}$. But~$\lambda^{-}(\sigma,t-1) = (\nu_1,\ldots,\nu_{t-1},\nu_{t+1},\ldots,\nu_k)$ and $\mu^{+}(\sigma,t-1) = \nu$ so Corollary~\ref{cor:schurpos2} indeed verifies this Schur inequality.

\bibliography{interlacing}{}
\bibliographystyle{plain}

\end{document}